\documentclass{amsart}

\usepackage{amsthm}
\usepackage[leqno]{amsmath}
\usepackage{latexsym,amsfonts,amssymb}
\usepackage[all]{xy} \SelectTips{eu}{} \SilentMatrices
\usepackage{hyperref}
\usepackage{multirow}
\usepackage[usenames,dvipsnames]{color}

\newcommand{\numberseries}{\mdseries}   

\newlength{\thmtopspace}                
\newlength{\thmbotspace}                
\newlength{\thmheadspace}               
\newlength{\thmindent}                  

\renewcommand{\subparagraph}{\vspace*{\thmbotspace}}

\newtheoremstyle{bfupright head,slanted body}
                {\thmtopspace}{\thmbotspace}
                {\slshape}{\thmindent}{\bfseries}{.}{\thmheadspace}
                {{\numberseries \thmnumber{(#2) }}\thmnote{#3}}

\newtheoremstyle{bfupright head,upright body}
                {\thmtopspace}{\thmbotspace}
                {\upshape}{\thmindent}{\bfseries}{.}{\thmheadspace}
                {{\numberseries \thmnumber{(#2) }}\thmnote{#3}}

\newtheoremstyle{bfit head,upright body}
                {\thmtopspace}{\thmbotspace}
                {\upshape}{\thmindent}{\upshape}{.}{\thmheadspace}
                {{\numberseries\thmnumber{(#2) }}
                {\bfseries\itshape\thmnote{\negthickspace#3}}}

\newtheoremstyle{it head,upright body}
                {\thmtopspace}{\thmbotspace}
                {\upshape}{\thmindent}{\upshape}{.}{\thmheadspace}
                {{\numberseries\thmnumber{(#2) }}
                {\itshape\thmnote{\negthickspace#3}}}

\newtheoremstyle{fixed bf head,slanted body}
                {\thmtopspace}{\thmbotspace}{\slshape}
                {\thmindent}{\bfseries}{.}{\thmheadspace}
                {{\numberseries \thmnumber{(#2) }}\thmname{#1}\thmnote{ (#3)}}

\newtheoremstyle{fixed bf head,upright body}
                {\thmtopspace}{\thmbotspace}{\upshape}
                {\thmindent}{\bfseries}{.}{\thmheadspace}
                {{\numberseries \thmnumber{(#2) }}\thmname{#1}\thmnote{ (#3)}}

\newtheoremstyle{fixed bfit head,upright body}
                {\thmtopspace}{\thmbotspace}{\upshape}
                {\thmindent}{\bfseries\itshape}{.}{\thmheadspace}
                {{\numberseries \thmnumber{(#2) }}\thmname{#1}\thmnote{ (#3)}}

\newtheoremstyle{sc head,small body}
                {\thmtopspace}{\thmbotspace}
                {\small\upshape}{\thmindent}{\scshape}{.}{\thmheadspace}
                {\thmname{#1}}

\newtheoremstyle{numbered paragraph}
                {\thmtopspace}{\thmbotspace}{\upshape}
                {\thmindent}{\upshape}{}{0pt}
                {{\numberseries \thmnumber{(#2) }}}

\newtheoremstyle{unnumbered paragraph}
                {\thmtopspace}{\thmbotspace}{\upshape}
                {\parindent}{\upshape}{}{0pt}

\theoremstyle{bfupright head,slanted body}
\newtheorem{res}{}[section]             \newtheorem*{res*}{}

\theoremstyle{bfit head,upright body}
                 \newtheorem*{com*}{}

\theoremstyle{bfupright head,upright body}
\newtheorem{bfhpg}[res]{}               \newtheorem*{bfhpg*}{}

\theoremstyle{it head,upright body}
               \newtheorem*{ithpg*}{}

\theoremstyle{sc head,small body}

\theoremstyle{fixed bf head,slanted body}
\newtheorem{thm}[res]{Theorem}          \newtheorem*{thm*}{Theorem}
\newtheorem{prp}[res]{Proposition}      \newtheorem*{prp*}{Proposition}
\newtheorem{cor}[res]{Corollary}        \newtheorem*{cor*}{Corollary}
\newtheorem{lem}[res]{Lemma}            \newtheorem*{lem*}{Lemma}

\theoremstyle{fixed bf head,upright body}
\newtheorem{dfn}[res]{Definition}       \newtheorem*{dfn*}{Definition}
\newtheorem{con}[res]{Construction}     \newtheorem*{con*}{Construction}
\newtheorem{obs}[res]{Observation}      \newtheorem*{obs*}{Observation}
\newtheorem{rmk}[res]{Remark}           \newtheorem*{rmk*}{Remark}
\newtheorem{exa}[res]{Example}          \newtheorem*{exa*}{Example}
         \newtheorem*{exe*}{Exercise}
\newtheorem{stp}[res]{Setup}            \newtheorem{stp*}{Setup}
         \newtheorem*{qst*}{Question}

\theoremstyle{numbered paragraph}
\newtheorem{ipg}[res]{}

\theoremstyle{unnumbered paragraph}
\newtheorem{ipg*}{}

\newlength{\thmlistleft}        
\newlength{\thmlistright}       
\newlength{\thmlistpartopsep}   
\newlength{\thmlisttopsep}      
\newlength{\thmlistparsep}      
\newlength{\thmlistitemsep}     

\newcounter{eqc} 
  {\end{list}}%

\newcounter{prt}
\newenvironment{prt}{\begin{list}{\upshape (\alph{prt})}%
    {\usecounter{prt}%
      \setlength{\leftmargin}{\thmlistleft}%
      \setlength{\labelwidth}{\thmlistleft}%
      \setlength{\rightmargin}{\thmlistright}%
      \setlength{\partopsep}{\thmlistpartopsep}%
      \setlength{\topsep}{\thmlisttopsep}%
      \setlength{\parsep}{\thmlistparsep}%
      \setlength{\itemsep}{\thmlistitemsep}}}%
  {\end{list}}%

\newcounter{rqm}
\newenvironment{rqm}{\begin{list}{\upshape (\arabic{rqm})}%
    {\usecounter{rqm}%
      \setlength{\leftmargin}{\thmlistleft}%
      \setlength{\labelwidth}{\thmlistleft}%
      \setlength{\rightmargin}{\thmlistright}%
      \setlength{\partopsep}{\thmlistpartopsep}%
      \setlength{\topsep}{\thmlisttopsep}%
      \setlength{\parsep}{\thmlistparsep}%
      \setlength{\itemsep}{\thmlistitemsep}}}%
  {\end{list}}%

  {\end{list}}%

\newenvironment{itemlist}{\nopagebreak \begin{list}{$\bullet$}%
    {\setlength{\leftmargin}{\thmlistleft}%
      \setlength{\labelwidth}{\thmlistleft}%
      \setlength{\rightmargin}{\thmlistright}%
      \setlength{\partopsep}{\thmlistpartopsep}%
      \setlength{\topsep}{\thmlisttopsep}%
      \setlength{\parsep}{\thmlistparsep}%
      \setlength{\itemsep}{\thmlistitemsep}}}%
  {\end{list}}%

\newenvironment{prf*}[1][Proof]{%
  \begin{proof}[\bf #1]
    \setcounter{equation}{0}
    \renewcommand{\theequation}{\arabic{equation}}}
  {\end{proof}
}

  \newcommand{\proofof}[2][:]{``#2''#1}

\newcommand{\pgref}[1]{(\ref{#1})}
\newcommand{\resref}[2][]{#1\pgref{res:#2}}
\newcommand{\stpref}[2][Setup~]{#1\pgref{stp:#2}}
\newcommand{\thmref}[2][Theorem~]{#1\pgref{thm:#2}}
\newcommand{\corref}[2][Corollary~]{#1\pgref{cor:#2}}
\newcommand{\prpref}[2][Proposition~]{#1\pgref{prp:#2}}
\newcommand{\lemref}[2][Lemma~]{#1\pgref{lem:#2}}
\newcommand{\obsref}[2][Observation~]{#1\pgref{obs:#2}}
\newcommand{\conref}[2][Construction~]{#1\pgref{con:#2}}
\newcommand{\dfnref}[2][Definition~]{#1\pgref{dfn:#2}}
\newcommand{\exaref}[2][Example~]{#1\pgref{exa:#2}}
\newcommand{\rmkref}[2][Remark~]{#1\pgref{rmk:#2}}

\newcommand{\secref}[2][Section~]{#1\ref{sec:#2}}

\renewcommand{\eqref}[1]{\pgref{eq:#1}}

\newcommand{\lemcite}[2][?]{\cite[lem.~#1]{#2}}

\makeatletter
\def\@nobreak@#1{\mathchoice%
  {\nobreakdef@\displaystyle\f@size{#1}}%
  {\nobreakdef@\nobreakstyle\tf@size{\firstchoice@false #1}}%
  {\nobreakdef@\nobreakstyle\sf@size{\firstchoice@false #1}}%
  {\nobreakdef@\nobreakstyle\ssf@size{\firstchoice@false #1}}%
  \check@mathfonts}%
\def\nobreakdef@#1#2#3{\hbox{{%
                    \everymath{#1}%
                    \let\f@size#2\selectfont%
                    #3}}}%
\makeatother

\numberwithin{equation}{res}

\newcommand{\cA}{\mathcal{A}}
\newcommand{\cB}{\mathcal{B}}
\newcommand{\cC}{\mathcal{C}}
\newcommand{\cM}{\mathcal{M}}

\newcommand{\m}{\mathfrak{m}}

\newcommand{\B}{\mathrm{B}}

\newcommand{\G}{\mathrm{G}}
\newcommand{\K}{\mathrm{K}}
\newcommand{\M}{\mathrm{M}}
\newcommand{\Q}{\mathrm{Q}}

\newcommand{\add}{\mathsf{add}}
\newcommand{\Ab}{\mathsf{Ab}}
\newcommand{\MCM}{\mathsf{MCM}}
\newcommand{\proj}{\mathsf{proj}}
\newcommand{\Mod}{\mathsf{Mod}}
\newcommand{\Aut}{\operatorname{Aut}}

\newcommand{\GL}{\mathrm{GL}}
\newcommand{\SL}{\mathrm{SL}}
\newcommand{\End}{\operatorname{End}}
\newcommand{\Hom}{\operatorname{Hom}}
\newcommand{\Ker}{\operatorname{Ker}}
\newcommand{\Coker}{\operatorname{Coker}}
\newcommand{\J}{\mathrm{J}}
\renewcommand{\Im}{\operatorname{Im}}
\renewcommand{\mod}{\mathsf{mod}}

\newcommand{\new}[1]{#1}

\def\widebardisplay#1{%
  \setbox0=\hbox{$\displaystyle #1$}
  \dimen0=\wd0%
  \advance\dimen0 by -4pt
  \vbox{%
    \nointerlineskip%
    \moveright 2pt 
    \vbox{\hrule width \dimen0}%
    \nointerlineskip%
    \kern 1pt
    \box0%
    }%
  }

\def\widebartext#1{%
  \setbox0=\hbox{$#1$}
  \dimen0=\wd0%
  \advance\dimen0 by -4pt
  \vbox{%
    \nointerlineskip%
    \moveright 2pt 
    \vbox{\hrule width \dimen0}%
    \nointerlineskip%
    \kern 1pt
    \box0%
    }%
  }

\def\widebarscript#1{%
  \setbox0=\hbox{$\scriptstyle #1$}
  \dimen0=\wd0%
  \advance\dimen0 by -3pt
  \vbox{%
    \nointerlineskip%
    \moveright 1.5pt 
    \vbox{\hrule width \dimen0}%
    \nointerlineskip%
    \kern .8pt
    \box0%
    }%
  }

\def\widebarscriptscript#1{%
  \setbox0=\hbox{$\scriptscriptstyle #1$}
  \dimen0=\wd0%
  \advance\dimen0 by -2pt
  \vbox{%
    \nointerlineskip%
    \moveright 1pt 
    \vbox{\hrule width \dimen0}%
    \nointerlineskip%
    \kern .6pt
    \box0%
    }%
  }

\begin{document}

\title{K-groups for rings of finite Cohen--Macaulay type}

\author{Henrik Holm}

\address{%
  Department of Mathematical Sciences\\
  Faculty of Science\\
  University of Copenhagen\\  
  Universitetsparken 5\\
  2100 Copenhagen {\O}\\
  Denmark}

\email{holm@math.ku.dk}
\urladdr{http://www.math.ku.dk/\~{}holm/} 

\keywords{Auslander--Reiten sequence; Bass' universal determinant
  group; finite Cohen--Macaulay type; maximal Cohen--Macaulay module;
  Quillen's $\K$-theory}

\subjclass[2010]{13C14, 13D15, 19B28}

\begin{abstract}
  For a local Cohen--Macaulay ring $R$ of finite CM-type, Yoshino has
  applied methods of Auslander and Reiten to compute the
  Gro\-then\-dieck group $\K_0$ of the category $\mod\,R$ of finitely
  generated $R$-modules. For the same type of rings, we compute in
  this paper the first Quillen $\K$-group $\K_1(\mod\,R)$.  We also
  describe the group homomorphism \mbox{$R^* \to \K_1(\mod\,R)$}
  induced by the inclusion functor \mbox{$\proj\,R \to \mod\,R$} and
  illustrate our results with concrete examples.
\end{abstract}

\maketitle

\section{Introduction}
\label{sec:Introduction}

Throughout this introduction, $R$ denotes a commutative noetherian
local Cohen--Macaulay ring.  The lower $\K$-groups of $R$ are known:
\mbox{$\K_0(R) \cong \mathbb{Z}$} and \mbox{$\K_1(R) \cong R^*$}. For
\mbox{$n \in \{0,1\}$} the classical $\K$-group $\K_n(R)$ of the ring
coincides with Quillen's $\K$-group $\K_n(\proj\,R)$ of the exact
category of finitely generated projective $R$-modules; and if $R$ is
regular, then Quillen's resolution theorem shows that the inclusion
functor $\proj\,R \to \mod\,R$ induces an isomorphism $\K_n(\proj\,R)
\cong \K_n(\mod\,R)$. If $R$ is non-regular, then these groups are
usually not isomorphic. The groups $\K_n(\mod\,R)$ are often denoted
$\G_n(R)$ and they are classical objects of study called the
$\G$-theory of $R$.  A celebrated result of Quillen is that
$\G$-theory is well-behaved under (Laurent) polynomial extensions:
$\G_n(R[t]) \cong \G_n(R)$ and $\G_n(R[t,t^{-1}]) \cong \G_n(R) \oplus
\G_{n-1}(R)$.

Auslander and Reiten \cite{MR816889} and Butler \cite{Butler} computed
$\K_0(\mod\,\Lambda)$ for an Artin algebra $\Lambda$ of finite
representation type.  Using similar techniques, Yoshino \cite{yos}
computed $\K_0(\mod\,R)$ in the case where $R$ has finite (as opposed
to tame or wild) CM-type:

\begin{thm*}[{Yoshino \cite[\bf thm.~(13.7)]{yos}}]
  \label{thm:0}
  Assume that $R$ is henselian and that it has a dualizing module. If
  $R$ has finite CM-type, then there is a group isomorphism,
  \begin{displaymath}
    \K_0(\mod\,R) \cong \Coker \Upsilon\;,  
  \end{displaymath}
  where \mbox{$\Upsilon \colon \mathbb{Z}^t \to \mathbb{Z}^{t+1}$} is
  the Auslander--Reiten homomorphism from \dfnref[]{AR-matrix}.
\end{thm*}

We mention that Yoshino's result is as much a contribution to algebraic $\K$-theory as
it is to the representation theory of the category $\MCM\,R$ of
maximal Cohen--Macaulay $R$-modules. Indeed, the inclusion functor \mbox{$\MCM\,R \to \mod\,R$} induces an
isomorphism $\K_n(\MCM\,R) \cong \K_n(\mod\,R)$ for every $n$. The theory of maximal Cohen--Macaulay modules, which originates from algebraic geometry and integral representations of finite groups, is a highly active area of research.

In this paper, we build upon results and techniques of Auslander and
Reiten \cite{MR816889}, Bass \cite{MR0249491}, Lam \cite{MR1838439},
Leuschke \cite{GJL07}, Quillen \cite{MR0338129}, Vaserstein
\cite{MR0267009, MR2111217}, and Yoshino \cite{yos} to compute the
group $\K_1(\mod\,R)$ when $R$ has finite CM-type.  Our main result is
\thmref{1}; it asserts that there is an isomorphism,
\begin{displaymath}
  \K_1(\mod\,R) \cong \Aut_R(M)_\mathrm{ab}/\Xi\;,
\end{displaymath}
where $M$ is any representation generator of the category of maximal
Cohen--Macau\-lay $R$-modules and $\Aut_R(M)_\mathrm{ab}$ is the
abelianization of its automorphism group. The subgroup $\Xi$ is more
complicated to describe; it is determined by the Auslander--Reiten
sequences and defined in \dfnref[]{Xi}. Observe that in contrast to $\K_0(\mod\,R)$, the group $\K_1(\mod\,R)$ is usually not finitely generated.

We also prove that if one
writes \mbox{$M = R \oplus M'$}, then the group homomorphism
$R^* \cong \K_1(\proj\,R) \to \K_1(\mod\,R)$ induced by the inclusion
functor \mbox{$\proj\,R \to \mod\,R$} can be identified with the map
\begin{displaymath}
  \lambda \colon R^* \longrightarrow \Aut_R(M)_{\mathrm{ab}}/\Xi
    \qquad \text{given by} \qquad
    r \longmapsto
    \begin{pmatrix}
      r1_R & 0 \\
      0 & 1_{M'}
    \end{pmatrix}.
\end{displaymath}

The paper is organized as follows: In \secref{main} we formulate our
main result, \thmref{1}. This theorem is not proved until
\secref{proof}, and the intermediate \secref[Sections~]{K} (on the Gersten--Sherman transformation), \secref[]{coherent} (on
Auslander's and Reiten's theory for coherent pairs),
\secref[]{semilocal} (on Vaserstein's result for semilocal rings),
\secref[]{functors} (on certain equivalences of categories), and
\secref[]{Y} (on Yoshino's results for the abelian category
$\mathcal{Y}$) prepare the ground.

In \secref[Sections~]{Aut-ab} and \secref[]{examples} we apply our
main theorem to compute the group $\K_1(\mod\,R)$ and the homomorphism
$\lambda \colon R^* \to \K_1(\mod\,R)$ in some concrete examples.
E.g.~for the simple curve singularity \mbox{$R =
  k[\hspace*{-1.4pt}[T^2,T^3]\hspace*{-1.4pt}]$} we obtain
\mbox{$\K_1(\mod\,R) \cong k[\hspace*{-1.4pt}[T]\hspace*{-1.4pt}]^*$}
and show that the homomorphism \mbox{$\lambda \colon
  k[\hspace*{-1.4pt}[T^2,T^3]\hspace*{-1.4pt}]^* \to
  k[\hspace*{-1.4pt}[T]\hspace*{-1.4pt}]^*$} is the inclusion.  It is
well-known that if $R$ is artinian with residue field $k$, then one
has \mbox{$\K_1(\mod\,R)\cong k^*$}. We apply \thmref{1} to confirm
this isomorphism for the ring \mbox{$R=k[X]/(X^2)$} of dual numbers
and to show that the homomorphism $\lambda \colon R^* \to k^*$ is
given by $a+bX \mapsto a^2$.

\new{We end this introduction by mentioning a related preprint \cite{Navkal} of Navkal. Although the present work and the paper of Navkal have been written completely independently (this fact is also pointet out in the latest version of \cite{Navkal}), there is a significant overlap between the two manuscripts: Navkal's main result \cite[thm.~1.2]{Navkal} is the existence of a long exact sequence involving the $\G$-theory of the rings $R$ and $\End_R(M)^\mathrm{op}$ (where $M$ is a particular representation generator of the category of maximal Cohen--Macaulay $R$-modules) and the $\K$-theory of certain division rings. In Section 5 in \emph{loc.\,cit.}, Navkal applies his main result to give some description of the group $\K_1(\mod\,R)$ for the ring \mbox{$R =  k[\hspace*{-1.4pt}[T^2,T^{2n+1}]\hspace*{-1.4pt}]$} where $n\geqslant 1$. We point out that the techniques used in this paper and in Navkal's work are quite different.}

\section{Formulation of the Main Theorem}
\label{sec:main}

Let $R$ be a commutative noetherian local Cohen--Macaulay ring.  By
$\mod\,R$~we denote the abelian category of finitely generated
$R$-modules. The exact categories of finitely generated projective
modules and of maximal Cohen--Macaulay mo\-du\-les over $R$ are
written $\proj\,R$ and $\MCM\,R$, respectively.  The goal of this
section is to state our main \thmref{1}; its proof is postponed to
\secref{proof}.

\begin{stp}
  \label{stp:assumptions}
  Throughout this paper, $(R,\m,k)$ is a commutative noetherian local
  Cohen--Macaulay ring satisfying the following assumptions.
  \begin{rqm}
  \item $R$ is henselian.
  \item $R$ admits a dualizing module.
  \item $R$ has \emph{finite CM-type}, that is, up to isomorphism,
    there are only finitely many non-isomorphic indecomposable maximal
    Cohen--Macaulay $R$-modules.
  \end{rqm}

  Note that (1) and (2) hold if $R$ is $\m$-adically complete.  Since
  $R$ is henselian, the ca\-te\-go\-ry $\mod\,R$ is Krull--Schmidt by
  \cite[prop.~(1.18)]{yos}; this fact will be important a number of
  times in this paper.

  Set \mbox{$M_0=R$} and let \mbox{$M_1,\ldots,M_t$} be a set of
  representatives for the isomorphism classes of non-free
  indecomposable maximal Cohen--Macaulay $R$-modules.  
Let $M$ be any
  \emph{representation generator} of $\MCM\,R$, that is, a finitely
  generated $R$-module such that $\add_RM = \MCM\,R$ \new{(where $\add_RM$ denotes the category of $R$-modules that are isomorphic to a direct summand of some finite direct sum of copies of $M$)}. For example, $M$ could be the square-free module
  \begin{equation}
    \label{eq:M}
    M = M_0 \oplus M_1 \oplus \cdots \oplus M_t\;.
  \end{equation}
  We denote by \mbox{$E = \End_R(M)$} the endomorphism ring of $M$.

  It follows from \cite[thm.~(4.22)]{yos} that $R$ is an isolated
  singularity, and hence by \emph{loc.\,cit.} thm.~(3.2) the category
  $\MCM\,R$ admits Auslan\-der--Reiten sequences. Let
  \begin{equation}
    \label{eq:AR}
    0 \longrightarrow \tau(M_j) \longrightarrow X_j 
    \longrightarrow M_j \longrightarrow 0
    \qquad (1 \leqslant j \leqslant t)
  \end{equation}
  be the Auslander--Reiten sequence in $\MCM\,R$ ending in $M_j$,
  where $\tau$ is the Aus\-lan\-der--Reiten translation.
\end{stp}

\begin{rmk}
  The one-dimensional Cohen--Macaulay rings of finite CM-type are classified by Cimen \cite{Cimen1,Cimen2}, Drozd and Ro{\u\i}ter \cite{DrozdRoiter}, Green and Reiner~\cite{GreenReiner}, and Wiegand \cite{Wiegand1,Wiegand2}. The two-dimensional complete Cohen--Macaulay rings of finite CM-type that contains the  complex numbers are classified by Auslander~\cite{MR816307}, Esnault \cite{Esnault}, and Herzog \cite{Herzog}.
 They are the invariant rings \mbox{$R=\mathbb{C}[\hspace*{-1.4pt}[X,Y]\hspace*{-1.4pt}]^G$} where $G$ is a non-trivial finite subgroup of $\GL_2(\mathbb{C})$. In this case, \mbox{$M=\mathbb{C}[\hspace*{-1.4pt}[X,Y]\hspace*{-1.4pt}]$} is a re\-pre\-sentation generator for $\MCM\,R$ which, unlike the one in \eqref{M}, need not be square-free.
\end{rmk}

\begin{dfn}
  \label{dfn:AR-matrix}
  For each Auslander--Reiten sequence \eqref{AR} we have
  \begin{displaymath}
    X_j \cong M_0^{n_{0j}} \oplus M_1^{n_{1j}} \oplus \cdots
    \oplus M_t^{n_{tj}}
  \end{displaymath}
  for uniquely determined $n_{0j},n_{1j},\ldots,n_{tj} \geqslant 0$.
  Consider the element,
  \begin{displaymath}
    \tau(M_j) + M_j - n_{0j}M_0 - n_{1j}M_1 - \cdots - n_{tj}M_t\;,
  \end{displaymath}
  in the free abelian group $\mathbb{Z}M_0 \oplus \mathbb{Z}M_1 \oplus
  \cdots \oplus \mathbb{Z}M_t$, and write this element as,
  \begin{displaymath}
    y_{0j}M_0 + y_{1j}M_1 + \cdots + y_{tj}M_t\;,
  \end{displaymath}
  where $y_{0j},y_{1j},\ldots,y_{tj} \in \mathbb{Z}$. Define the
  \emph{Auslander--Reiten matrix} $\Upsilon$ as the $(t+1) \times t$
  matrix with entries in $\mathbb{Z}$ whose $j$'th column is
  $(y_{0j},y_{1j},\ldots,y_{tj})$. When $\Upsilon$ is viewed as a
  homomorphism of abelian groups $\Upsilon \colon \mathbb{Z}^t \to
  \mathbb{Z}^{t+1}$ (elements in $\mathbb{Z}^t$ and $\mathbb{Z}^{t+1}$
  are viewed as column vectors), we refer to it as the
  \emph{Auslander--Reiten homomorphism}.
\end{dfn}

\begin{exa}
  \label{exa:7x6}
  Let
  \mbox{$R=\mathbb{C}[\hspace*{-1.4pt}[X,Y,Z]\hspace*{-1.4pt}]/(X^3+Y^4+Z^2)$}.
  Besides \mbox{$M_0=R$} there are exactly \mbox{$t=6$} non-isomorphic
  indecomposable maximal Cohen--Macaulay modules, and the
  Aus\-lan\-der--Reiten sequences have the following form,
  \begin{eqnarray*}
    0 \longrightarrow M_1 \longrightarrow M_2 \longrightarrow M_1 \longrightarrow 0\;\phantom{.} \\
    0 \longrightarrow M_2 \longrightarrow M_1 \oplus M_3 \longrightarrow M_2 \longrightarrow 0\;\phantom{.} \\    
    0 \longrightarrow M_3 \longrightarrow M_2 \oplus M_4 \oplus M_6 \longrightarrow M_3 \longrightarrow 0\;\phantom{.} \\    
    0 \longrightarrow M_4 \longrightarrow M_3 \oplus M_5 \longrightarrow M_4 \longrightarrow 0\;\phantom{.} \\    
    0 \longrightarrow M_5 \longrightarrow M_4 \longrightarrow M_5 \longrightarrow 0\;\phantom{.} \\    
    0 \longrightarrow M_6 \longrightarrow M_0 \oplus M_3 \longrightarrow M_6 \longrightarrow 0\;;
  \end{eqnarray*}
  see \cite[(13.9)]{yos}.  The $7\!\times\!6$ Auslander--Reiten matrix
  $\Upsilon$ is therefore given by
  \begin{displaymath}
    \Upsilon = 
    \begin{pmatrix}
       0 &  0 &  0 &  0 &  0 & -1 \\
       2 & -1 &  0 &  0 &  0 &  0 \\
      -1 &  2 & -1 &  0 &  0 &  0 \\
       0 & -1 &  2 & -1 &  0 & -1 \\
       0 &  0 & -1 &  2 & -1 &  0 \\
       0 &  0 &  0 & -1 &  2 &  0 \\
       0 &  0 & -1 &  0 &  0 &  2
    \end{pmatrix}.
  \end{displaymath}
  In this case, the Auslander--Reiten homomorphism $\Upsilon \colon
  \mathbb{Z}^6 \to \mathbb{Z}^7$ is clearly injective.
\end{exa}

One hypothesis in our main result, \thmref{1} below, is that the
Auslander--Reiten homomorpism $\Upsilon$ over the ring $R$ in question
is injective. We are not aware of an example where $\Upsilon$ is not
injective. The following lemma covers the situation of the rational
double points, that is, the invariant rings
\mbox{$R=k[\hspace*{-1.4pt}[X,Y]\hspace*{-1.4pt}]^G$}, where $k$ is an
algebraically closed field of characteristic $0$ and $G$ is a
non-trivial finite subgroup of $\SL_2(k)$; see \cite{AR}.

\begin{lem}
  Assume that $R$ is complete, integrally closed, non-regular,
  Gorenstein, of Krull dimension $2$, and that the residue field $k$
  is algebraically closed.  Then the Auslander--Reiten homomorphism
  $\Upsilon$ is injective.
\end{lem}

\begin{proof}
  Let \mbox{$1 \leqslant j \leqslant t$} be given and consider the
  expression
  \begin{displaymath}
    \tau(M_j) + M_j - n_{0j}M_0 - n_{1j}M_1 - \cdots - n_{tj}M_t
    \,=\,
    y_{0j}M_0 + y_{1j}M_1 + \cdots + y_{tj}M_t
  \end{displaymath}
  in the free abelian group \mbox{$\mathbb{Z}M_0 \oplus \mathbb{Z}M_1
    \oplus \cdots \oplus \mathbb{Z}M_t$}, see \dfnref{AR-matrix}.  Let
  $\Gamma$ be the Auslander--Reiten quiver of $\MCM\,R$.  We recall from
  \cite[thm.~1]{AR} that the arrows in $\Gamma$ occur in pairs
  \mbox{\smash{$\!\xymatrix@R=-0.5pc@C=4ex{\circ \ar@<0.5ex>[r] & \circ
      \ar@<0.5ex>[l]}$}}\!, and that collapsing each pair to an
  undirected edge gives an extended Dynkin diagram
  \smash{$\tilde{\Delta}$}.  Moreover, removing the vertex
  corresponding to $M_0=R$ and any incident edges gives a Dynkin graph
  $\Delta$.

  Now, $X_j$ has a direct summand $M_k$ if and only if there is an
  arrow $M_k \rightarrow M_j$ in $\Gamma$.  Also, the
  Auslander--Reiten translation $\tau$ satisfies $\tau(M_j) = M_j$ by
  \cite[proof of thm.\ 1]{AR}.  Combined with the structure of the
  Auslander--Reiten quiver, this means~that
  \begin{displaymath}
    y_{kj} =
    \left\{
      \begin{array}{cl}
        2  & \mbox{ if $k = j$\;,} \\
        -1 & \mbox{ if there is an edge 
          \smash{$\!\xymatrix@C=1.2pc{M_k \ar@{-}[r] & M_j}\!$} in $\tilde{\Delta}$\;,} \\
        0  & \mbox{ otherwise\;.}
      \end{array}
    \right.    
  \end{displaymath}
  Hence the \mbox{$t \times t$} matrix $\Upsilon_0$ with $(y_{1j},
  \ldots, y_{tj})$ as $j$'th column, where \mbox{$1 \leqslant j
    \leqslant t$}, is the Cartan matrix of the Dynkin graph $\Delta$;
  cf.\ \cite[def.\ 4.5.3]{Benson}.  This matrix is invertible by
  \cite[exer.\ (21.18)]{FuHa}.  Deleting the first row
  $(y_{01},\ldots,y_{0t})$ in the Auslander--Reiten matrix $\Upsilon$,
  we get the invertible matrix $\Upsilon_0$, and consequently,
  \mbox{$\Upsilon \colon \mathbb{Z}^t \to \mathbb{Z}^{t+1}$}
  determines an injective homomorphism.
\end{proof}

For a group $G$ we denote by $G_\mathrm{ab}$ its
\emph{abelianization}, i.e.~\mbox{$G_\mathrm{ab} = G/[G,G]$}, where
$[G,G]$ is the commutator subgroup of $G$. 

We refer to the following as the \emph{tilde construction}. It
associates to every automorphism $\alpha \colon X \to X$ of a maximal
Cohen--Macaulay module $X$ an automorphism $\tilde{\alpha} \colon M^q
\to M^q$ of the smallest power $q$ of the representation generator $M$
such that $X$ is a direct summand of $M^q$.

\begin{bfhpg}[Construction]
  \label{res:tilde}
  The chosen representation generator $M$ for $\MCM\,R$ has the form
  \mbox{$M = M_0^{m_0} \oplus \cdots \oplus M_t^{m_t}$} for uniquely
  determined integers \mbox{$m_0,\ldots,m_t>0$}.
  For any module \mbox{$X = M_0^{n_0} \oplus \cdots \oplus
    M_t^{n_t}$} in $\MCM\,R$, we define natural numbers,
  \begin{align*}
    q &= q(X) = \min\{p \in \mathbb{N} \,|\, pm_j \geqslant n_j
    \text{ for all } 0 \leqslant j \leqslant t \}\;, \ \text{and} \\
    v_j &= v_j(X) = qm_j-n_j \geqslant 0\;,
  \end{align*}
  and a module \mbox{$Y = M_0^{v_0} \oplus \cdots \oplus M_t^{v_t}$}
  in $\MCM\,R$.  Let \mbox{$\psi \colon X \oplus Y
    \stackrel{\cong}{\longrightarrow} M^q$} be the $R$-iso\-mor\-phism
  that maps an element
  \begin{align*}
    ((\underline{x}_0,\ldots,\underline{x}_t),
    (\underline{y}_0,\ldots,\underline{y}_t)) 
    \in X \oplus Y = 
    (M_0^{n_0} \oplus \cdots \oplus M_t^{n_t}) \oplus 
    (M_0^{v_0} \oplus \cdots \oplus M_t^{v_t})\;,
  \end{align*}
  where $\underline{x}_j \in M_j^{n_j}$ and $\underline{y}_j \in
  M_j^{v_j}$, to the element
  \begin{align*}
    ((\underline{z}_{01},\ldots,\underline{z}_{t1}), \ldots,
  (\underline{z}_{0q},\ldots,\underline{z}_{tq})) \in M^q = 
    (M_0^{m_0} \oplus \cdots \oplus M_t^{m_t})^q\;, 
  \end{align*}
  where $\underline{z}_{j1},\ldots,\underline{z}_{jq} \in M_j^{m_j}$
  are given by $(\underline{z}_{j1},\ldots,\underline{z}_{jq}) =
  (\underline{x}_j,\underline{y}_j) \in M_j^{qm_j} = M_j^{n_j+v_j}$.

  Now, given $\alpha$ in $\Aut_R(X)$, we define $\tilde{\alpha}$ to be
  the uniquely determined element in $\Aut_R(M^q)$ that makes the
  following diagram commutative,
  \begin{displaymath}
    \xymatrix{
      X \oplus Y
      \ar[r]^-{\psi}_-{\cong} \ar[d]_-{\alpha \oplus 1_Y}^-{\cong} & 
      M^q\,\phantom{.} \ar@<-2pt>[d]^-{\tilde{\alpha}}_-{\cong} \\
      X \oplus Y
      \ar[r]_-{\psi}^-{\cong} & M^q\,.
    }
  \end{displaymath}
  The automorphism $\tilde{\alpha}$ of $M^q$ has the form
  $\tilde{\alpha}=(\tilde{\alpha}_{ij})$ for uniquely determined
  endomorphisms $\tilde{\alpha}_{ij}$ of $M$, that is,
  $\tilde{\alpha}_{ij} \in E = \End_R(M)$. Hence
  $\tilde{\alpha}=(\tilde{\alpha}_{ij})$ can naturally be viewed as an
  invertible $q \times q$ matrix with entries in $E$.
\end{bfhpg}

\begin{exa}
  \label{exa:diag}
  Let $M=M_0 \oplus \cdots \oplus M_t$ and $X = M_j$. Then $q=1$ and
  \begin{displaymath}
    Y = M_0 \oplus \cdots \oplus M_{j-1} \oplus M_{j+1} 
    \oplus \cdots \oplus M_t\;.
  \end{displaymath}
  The isomorphism \mbox{$\psi \colon X \oplus Y \to M$} maps
  $(x_j,(x_0,\ldots,x_{j-1},x_{j+1},\ldots,x_t))$ in \mbox{$X \oplus
    Y$} to $(x_0,\ldots,x_{j-1},x_j,x_{j+1},\ldots,x_t)$ in $M$.
  Therefore, for \mbox{$\alpha \in \Aut_R(X) = \Aut_R(M_j)$},
  Construction \resref{tilde} yields the following automorphism of
  $M$,
  \begin{displaymath}
    \tilde{\alpha} = \psi(\alpha \oplus
    1_Y)\psi^{-1} = 1_{M_0} \oplus \cdots \oplus 1_{M_{j-1}} \oplus \alpha \oplus 1_{M_{j+1}} \oplus \cdots \oplus 1_{M_t}\;,
  \end{displaymath}
  which is an invertible \mbox{$1 \times 1$} matrix with entry
  in $E=\End_R(M)$.
\end{exa}

The following result on Auslander--Reiten sequences is quite standard.
We provide a few proof details along with the appropriate references.

\begin{prp}
  \label{prp:AR}
  Let there be given Auslander--Reiten
  sequences in $\MCM\,R$,
  \begin{align*}
    0 \longrightarrow \tau(M) \longrightarrow X \longrightarrow M \longrightarrow 0
    \quad \text{ and } \quad
    0 \longrightarrow \tau(M') \longrightarrow X' \longrightarrow M' \longrightarrow 0\;.
  \end{align*}
  If \mbox{$\alpha \colon M \to M'$} is a homomorphism, then there
  exist homomorphisms $\beta$ and $\gamma$ that make the following
  diagram commutative,
  \begin{displaymath}
    \xymatrix{
      0 \ar[r] & \tau(M) \ar@{-->}[d]^-{\gamma} \ar[r] & X \ar@{-->}[d]^-{\beta} \ar[r] & M \ar[d]^-{\alpha} \ar[r] & 0\,\phantom{.} \\
      0 \ar[r] & \tau(M') \ar[r] & X' \ar[r] & M' \ar[r] & 0\,.
    }
  \end{displaymath}
  Furthermore, if $\alpha$ is an isomorphism, then so are $\beta$ and
  $\gamma$.
\end{prp}

\begin{proof}
  Write \mbox{$\rho \colon X \to M$} and \mbox{$\rho' \colon X' \to
    M'$}. It suffices to prove the existence of $\beta$ such that
  $\rho' \beta = \alpha \rho$, because then the existence of $\gamma$
  follows from diagram chasing.

  As \mbox{$0 \to \tau(M') \to X' \to M' \to 0$} is an
  Auslander--Reiten sequence, it suffices by \cite[lem.~(2.9)]{yos} to
  show that $\alpha\rho \colon X \to M'$ is not a split epimorphism.
  Suppose that there do exist \mbox{$\tau \colon M' \to X$} with
  \mbox{$\alpha\rho\tau = 1_{M'}$}. Hence $\alpha$ is a split
  epimorphism.  As $M$ is indecomposable, $\alpha$ must be an
  isomorphism. Thus \mbox{$\rho\tau\alpha =
    \alpha^{-1}(\alpha\rho\tau)\alpha = 1_M$}, which contradicts the
  fact that $\rho$ is not a split epimorphism.

  The fact that $\beta$ and $\gamma$ are isomorphisms if $\alpha$ is
  so follows from \cite[lem.~(2.4)]{yos}.
\end{proof}

The choice requested in the following construction is possible by
\prpref{AR}.

\begin{con}
  \label{con:j-alpha}
  Choose for each \mbox{$1 \leqslant j \leqslant t$} and every
  \mbox{$\alpha \in \Aut_R(M_j)$} elements $\beta_{j,\alpha} \in
  \Aut_R(X_j)$ and $\gamma_{j,\alpha} \in \Aut_R(\tau(M_j))$ that make
  the next diagram commute,
  \begin{equation}
    \label{eq:j-alpha-diagram}
    \begin{gathered}
    \xymatrix{
      0 \ar[r] & \tau(M_j) \ar[d]_-{\cong}^-{\gamma_{j,\alpha}}
      \ar[r] & X_j \ar[d]_-{\cong}^-{\beta_{j,\alpha}} \ar[r] &
      M_j \ar[d]_-{\cong}^-{\alpha} \ar[r] & 0\,\phantom{.} \\
      0 \ar[r] & \tau(M_j) \ar[r] & X_j \ar[r] & M_j \ar[r] & 0\,;
    }
    \end{gathered}
  \end{equation}
  here the row(s) is the $j$'th Auslander--Reiten sequence \eqref{AR}.
\end{con}

As shown in \lemref{End-semilocal}, the endomorphism ring
$E=\End_R(M)$ of the chosen representation generator $M$ is semilocal,
that is, $E/\J(E)$ is semisimple. Thus, if the ground ring $R$, and
hence also the endomorphism ring $E$, is an algebra over the residue
field $k$ and \mbox{$\operatorname{char}(k) \neq 2$}, then a result by
Vaserstein \cite[thm.~2]{MR2111217} yields that the canonical
homomorphism \mbox{$\theta_E \colon E^*_{\mathrm{ab}} \to
  \K_1^{\mathrm{C}}(E)$} is an isomorphism. Here
$\K_1^{\mathrm{C}}(E)$ is the classical $\K_1$-group of the ring $E$;
see~\resref{classical}.  Its inverse,
\begin{equation*}
  \theta_E^{-1} = \textstyle{\det_E} \colon \K_1^{\mathrm{C}}(E) \longrightarrow E^*_\mathrm{ab} = \Aut_R(M)_{\mathrm{ab}}\;,
\end{equation*}
is called the \emph{generalized determinant map}. The details are
discussed in \secref{semilocal}. We are now in a position to define
the subgroup $\Xi$ of $\Aut_R(M)_{\mathrm{ab}}$ that appears in our
main \thmref{1} below.

\begin{dfn}
  \label{dfn:Xi}
  Let $(R,\m,k)$ be a ring satisfying the hypotheses in
  \stpref{assumptions}.  Assume, in addition, that $R$ is an algebra
  over $k$ and that one has \mbox{$\operatorname{char}(k) \neq 2$}.
  Define a subgroup $\Xi$ of $\Aut_R(M)_{\mathrm{ab}}$ as follows.
  \begin{itemlist}
  \item[--] Choose for each \mbox{$1 \leqslant j \leqslant t$} and
    each \mbox{$\alpha \in \Aut_R(M_j)$} elements $\beta_{j,\alpha}
    \in \Aut_R(X_j)$ and $\gamma_{j,\alpha} \in \Aut_R(\tau(M_j))$ as
    in \conref{j-alpha}.

  \item[--] Let \smash{$\tilde{\alpha}$, $\tilde{\beta}_{j,\alpha}$,
      and $\tilde{\gamma}_{j,\alpha}$} be the invertible matrices with
    entries in $E$ obtained by applying the tilde construction
    \resref{tilde} to \smash{$\alpha$, $\beta_{j,\alpha}$, and
      $\gamma_{j,\alpha}$}.
  \end{itemlist}
  Let $\Xi$ be the subgroup of $\Aut_R(M)_{\mathrm{ab}}$
  generated by the elements
  \begin{displaymath}
    \textstyle{(\det_E\tilde{\alpha})(\det_E\tilde{\beta}_{j,\alpha})^{-1}(\det_E\tilde{\gamma}_{j,\alpha})}\;,
  \end{displaymath}
  where $j$ ranges over $\{1,\ldots,t\}$ and $\alpha$ over
  $\Aut_R(M_j)$.
\end{dfn}

A priori the definition of the group $\Xi$ involves certain choices. However, it follows from \prpref{sigma} that $\Xi$ is actually independent of the choices made.

\begin{rmk}
  \label{rmk:remark-to-definition-Xi}
  In specific examples it is convenient to consider the simplest
  possible representation generator \mbox{$M=M_0\oplus M_1 \oplus
    \cdots \oplus M_t$}. In this case, \exaref{diag} shows that
  $\tilde{\alpha}$ and $\tilde{\gamma}_{j,\alpha}$ are
  \mbox{$1\!\times\!1$} matrices with entries in $E$, that is,
  \mbox{$\tilde{\alpha}, \tilde{\gamma}_{j,\alpha} \in E^*$}, and
  consequently \mbox{$\det_E\tilde{\alpha} = \tilde{\alpha}$} and
  \mbox{$\det_E\tilde{\gamma}_{j,\alpha} = \tilde{\gamma}_{j,\alpha}$}
  as elements in $E^*_\mathrm{ab}$.
\end{rmk}

We are now in a position to state our main result.

\begin{thm}
  \label{thm:1}
  Let $(R,\m,k)$ be a ring satisfying the hypotheses in
  \stpref{assumptions}.  Assume that $R$ is an algebra over its
  residue field $k$ with \mbox{$\operatorname{char}(k) \neq 2$}, and
  that the Auslander--Reiten homomorphism $\Upsilon \colon
  \mathbb{Z}^t \to \mathbb{Z}^{t+1}$ from \dfnref{AR-matrix} is
  injective.

  Let $M$ be any representation generator of \mbox{$\MCM\,R$}. There
  is an isomorphism,
  \begin{displaymath}
    \K_1(\mod\,R) \cong \Aut_R(M)_{\mathrm{ab}}/\Xi\;,
  \end{displaymath}
  where $\Xi$ is the subgroup of $\Aut_R(M)_{\mathrm{ab}}$ given in
  \dfnref{Xi}.  

  Furthermore, if \mbox{$\,\mathrm{inc} \colon \proj\,R
    \to \mod\,R$} is the inclusion functor and \mbox{$M=R
    \oplus M'$}, then \mbox{$\K_1(\mathrm{inc}) \colon \K_1(\proj\,R)
    \to \K_1(\mod\,R)$} may be identified with the homomorphism,
  \begin{displaymath}
    \lambda \colon R^* \longrightarrow \Aut_R(M)_{\mathrm{ab}}/\Xi
    \qquad \text{given by} \qquad
    r \longmapsto
    \begin{pmatrix}
      r1_R & 0 \\
      0 & 1_{M'}
    \end{pmatrix}.
  \end{displaymath}
\end{thm}

As mentioned in the introduction, the proof of \thmref{1} spans
\secref[Sections~]{K} to \secref[]{proof}. Applications and
examples are presented in \secref[Sections~]{Aut-ab} and
\secref[]{examples}. The interested reader could go ahead and read \secref[Sections~]{Aut-ab}--\secref[]{examples} right away, since these sections are practically independent of \secref[]{K}--\secref[]{proof}.

\section{The Gersten--Sherman Transformation}
\label{sec:K}

To prove \thmref{1}, we need to compare and/or identify various $\K$-groups. The relevant definitions and properties of these $\K$-groups are recalled below. The (so-called) Gersten--Sherman transformation is our most valuable tool for comparing $\K$-groups, and the main part of this section is devoted to this natural transformation. Readers who are familiar with $\K$-theory may skip this section altogether.

In the following, the Grothendieck group functor is denoted by $\G$.

\begin{ipg}
  \label{res:classical}
  Let $A$ be a unital ring. 

  The \emph{classical $\K_0$-group} of $A$ is defined as
  \mbox{$\K_0^{\mathrm{C}}(A)=\G(\proj\,A)$}, that is, the
  Gro\-then\-dieck group of the category of finitely generated
  projective $A$-modules.

  The \emph{classical $\K_1$-group} of $A$ is defined as
  $\K_1^{\mathrm{C}}(A)=\GL(A)_{\mathrm{ab}}$, that is, the
  abelianization of the infinite (or stable) general linear group; see
  e.g.~Bass \cite[chap.~V]{MR0249491}.
\end{ipg}

\begin{ipg}
  \label{res:loop-category}
  Let $\cC$ be any category. Its \emph{loop category} $\Omega\cC$ is
  the category whose objects are pairs $(C,\alpha)$ with $C \in \cC$
  and $\alpha \in \Aut_{\cC}(C)$. A morphism $(C,\alpha) \to
  (C',\alpha')$ in $\Omega\cC$ is a commutative diagram in $\cC$,
  \begin{displaymath}
    \xymatrix{
      C \ar[d]_-{\alpha}^-{\cong} \ar[r]^-{\psi} & 
      C'\phantom{.} \ar[d]^-{\alpha'}_-{\cong}\\
      C \ar[r]^-{\psi} & C'.
    }
  \end{displaymath}
\end{ipg}

\begin{ipg}
  \label{res:K1B}
  Let $\cC$ be a skeletally small exact category. Its loop category
  $\Omega\cC$ is also skeletally small, and it inherits a natural
  exact structure from $\cC$.  \emph{Bass' $\K_1$-group} (also called
  \emph{Bass' universal determinant group}) of $\cC$, which we denote
  by $\K_1^{\mathrm{B}}(\cC)$, is the Grothendieck group of
  $\Omega\cC$, that is $\G(\Omega\cC)$, modulo the subgroup generated
  by all elements of the form
  \begin{displaymath}
    (C,\alpha) + (C,\beta) - (C,\alpha\beta)\;,
  \end{displaymath}
  where \mbox{$C \in \cC$} and \mbox{$\alpha,\beta \in
    \Aut_{\cC}(C)$}; see the book of Bass \cite[chap.~VIII\S1]{MR0249491} or
  Rosenberg \cite[def.~3.1.6]{MR1282290}.  For $(C,\alpha)$ in
  $\Omega\cC$ we denote by $[C,\alpha]$ its image in
  $\K_1^{\mathrm{B}}(\cC)$.
\end{ipg}

\begin{ipg}
  \label{res:neutral}
  For every $C$ in $\cC$ one has $[C,1_C]+[C,1_C] = [C,1_C1_C] =
  [C,1_C]$ in $\K_1^{\mathrm{B}}(\cC)$. Consequently, $[C,1_C]$ is the
  neutral element in $\K_1^{\mathrm{B}}(\cC)$.
\end{ipg}

\begin{ipg}
  \label{res:free}
  For a unital ring $A$ there is by \cite[thm.~3.1.7]{MR1282290} a
  natural isomorphism,
  \begin{displaymath}
    \smash{\eta_A \colon \K_1^{\mathrm{C}}(A) \stackrel{\cong}{\longrightarrow}
    \K_1^{\mathrm{B}}(\proj\,A)\;.}
  \end{displaymath}

  The isomorphism $\eta_A$ maps \mbox{$\xi \in \GL_n(A)$}, to the
  class \mbox{$[A^n,\xi] \in \K_1^{\mathrm{B}}(\proj\,A)$}. Here $\xi$
  is viewed as an automorphism of the row space $A^n$ (a free left
  $A$-module), that is, $\xi$ acts by multiplication from the right.

  The inverse map $\eta_A^{-1}$ acts as follows.  Let $[P,\alpha]$ be
  in $\K_1^{\mathrm{B}}(\proj\,A)$. Choose any $Q$ in $\proj\,A$ and
  any isomorphism \smash{$\psi \colon P \oplus Q \to A^n$} with $n \in
  \mathbb{N}$. In $\K_1^{\mathrm{B}}(\proj\,A)$ one has
  \begin{displaymath}
    [P,\alpha] = [P,\alpha] + [Q,1_Q] = [P\oplus Q,\alpha \oplus 1_Q] 
    = [A^n,\psi(\alpha \oplus 1_Q)\psi^{-1}]\;.
  \end{displaymath}
  The automorphism $\psi(\alpha \oplus 1_Q)\psi^{-1}$ of (the row
  space) $A^n$ can be identified with a matrix in $\beta \in
  \GL_n(A)$. The action of \smash{$\eta_A^{-1}$} on $[P,\alpha]$ is
  now $\beta$'s image in \smash{$\K_1^{\mathrm{C}}(A)$}.
\end{ipg}

\begin{ipg}
  \label{res:Quillen}
  Quillen defines in \cite{MR0338129} functors $\K_n^{\mathrm{Q}}$
  from the category of skeletally small exact categories to the
  category of abelian groups.  More precisely,
  \mbox{$\K_n^{\mathrm{Q}}(\cC) = \pi_{n+1}(\B\Q\cC,0)$} where $\Q$ is
  Quillen's $\Q$-construction and $\B$ denotes the classifying space.

  The functor $\K_0^{\mathrm{Q}}$ is naturally isomorphic to the
  Grothendieck group functor $\G$; see \cite[\S2 thm.~1]{MR0338129}.
  For a ring $A$ there is a natural isomorphism
  \smash{$\K_1^{\mathrm{Q}}(\proj\,A) \cong \K_1^{\mathrm{C}}(A)$};
  see for example Srinivas \cite[cor.~(2.6) and
  thm.~(5.1)]{MR1382659}.
\end{ipg}

Gersten sketches in \cite[\S5]{MR0382398} the construction of a natural
transformation $\zeta \colon \K_1^{\mathrm{B}} \to \K_1^{\mathrm{Q}}$
of functors on the category of skeletally small exact categories.  The
details of this construction were later given by Sherman
\cite[\S3]{Sherman}, and for this reason we refer to $\zeta$ as the
\emph{Gersten--Sherman transformation}\footnote{\ In the papers by
  Gersten \cite{MR0382398} and Sherman \cite{Sherman}, the functor
  $\K_1^{\mathrm{B}}$ is denoted by $\K_1^{\mathrm{det}}$.}.  Examples
due to Gersten and Murthy \cite[prop.~5.1 and 5.2]{MR0382398} show
that for a general skeletally small exact category $\cC$, the
homomorphism \smash{$\zeta_{\cC} \colon \K_1^{\mathrm{B}}(\cC) \to
  \K_1^{\mathrm{Q}}(\cC)$} is neither injective nor surjective.  For
the exact category $\proj\,A$, where $A$ is a ring, it is known that
$\K_1^{\mathrm{B}}(\proj\,A)$ and $\K_1^{\mathrm{Q}}(\proj\,A)$ are
isomorphic, indeed, they are both isomorphic to the classical
$\K$-group $\K_1^{\mathrm{C}}(A)$; see~\resref{free} and
\resref{Quillen}. Therefore, a natural question arises: is
$\zeta_{\proj\,A}$ an isomorphism? Sherman answers this question
affirmatively in \cite[pp.~231--232]{Sherman}; in fact, in
\emph{loc.\,cit.}  Theorem~3.3 it is proved that $\zeta_{\cC}$ is an
isomorphism for every semisimple exact ca\-te\-gory, that is, an exact
category in which every short exact sequence splits. We note these
results of Gersten and Sherman for later use.

\begin{thm}
  \label{thm:Gersten-transformation-1}
  There exists a natural transformation \smash{$\zeta \colon
    \K_1^{\mathrm{B}} \to \K_1^{\mathrm{Q}}$}, which we call the
  \emph{\emph{Gersten--Sherman transformation}}, of functors on the
  category of skeletally small exact categories such that
  \smash{$\zeta_{\proj\,A} \colon \K_1^{\mathrm{B}}(\proj\,A) \to
    \K_1^{\mathrm{Q}}(\proj\,A)$} is an isomorphism for every ring
  $A$. \qed
\end{thm}

We will also need the next result on the Gersten--Sherman transformation.
Recall that a \emph{length category} is an abelian category in which
every object has finite length.

\begin{thm}
  \label{thm:Gersten-transformation-2}
  If $\cA$ is a skeletally small length category with only finitely
  many simple objects (up to isomorphism), then \smash{$\zeta_{\cA}
    \colon \K_1^{\mathrm{B}}(\cA) \to \K_1^{\mathrm{Q}}(\cA)$} is an
  isomorphism.
\end{thm}

\begin{proof}
  We begin with a general observation. Given skeletally small exact
  categories $\cC_1$ and $\cC_2$, there are exact projection functors
  \smash{$p_j \colon \cC_1 \times \cC_2 \to \cC_j$} (\mbox{$j =1,2$}).
  From the ``elementary properties'' of Quillen's $\K $-groups listed
  in \cite[\S 2]{MR0338129}, it follows that the homomorphism
  $(\K_1^{\mathrm{Q}}(p_1),\K_1^{\mathrm{Q}}(p_2)) \colon
  \K_1^{\mathrm{Q}}(\cC_1 \times \cC_2) \to \K_1^{\mathrm{Q}}(\cC_1)
  \oplus \K_1^{\mathrm{Q}}(\cC_2)$ is an isomorphism.  A similar
  argument shows that
  $(\K_1^{\mathrm{B}}(p_1),\K_1^{\mathrm{B}}(p_2))$ is an isomorphism.
  Since \smash{$\zeta \colon \K_1^{\mathrm{B}} \to \K_1^{\mathrm{Q}}$}
  is a natural transformation, it follows that $\zeta_{\cC_1 \times \cC_2}$ is an isomorphism if and
  only if $\zeta_{\cC_1}$ and $\zeta_{\cC_2}$ are isomorphisms.

  Denote by $\cA_\mathrm{ss}$ the full subcategory of $\cA$ consisting
  of all semisimple objects. Note that $\cA_\mathrm{ss}$ is a Serre
  subcategory of $\cA$, and hence $\cA_\mathrm{ss}$ is itself an
  abelian category. Let $i \colon \cA_\mathrm{ss} \hookrightarrow \cA$
  be the (exact) inclusion and consider the commutative diagram,
  \begin{displaymath}
    \xymatrix@C=10ex{
      \K_1^{\mathrm{B}}(\cA_\mathrm{ss}) 
      \ar[d]_-{\zeta_{\cA_\mathrm{ss}}}
      \ar[r]^-{\K_1^{\mathrm{B}}(i)}_-{\cong}
      & \K_1^{\mathrm{B}}(\cA)\,\phantom{.} 
      \ar@<-2pt>[d]^-{\zeta_{\cA}} \\
      \K_1^{\mathrm{Q}}(\cA_\mathrm{ss})
      \ar[r]^-{\K_1^{\mathrm{Q}}(i)}_-{\cong}
      & \K_1^{\mathrm{Q}}(\cA)\,.
    }
  \end{displaymath}
  Since $\cA$ is a length category, Bass' and Quillen's devissage
  theorems \cite[VIII\S3 thm. (3.4)(a)]{MR0249491} and \cite[\S5
  thm.~4]{MR0338129} show that $\K_1^{\mathrm{B}}(i)$ and
  $\K_1^{\mathrm{Q}}(i)$ are isomorphisms.  Hence, it suffices to
  argue that $\zeta_{\cA_\mathrm{ss}}$ is an isomorphism. By
  assumption there is a finite set $\{S_1,\ldots,S_n\}$ of
  representatives of the isomorphism classes of simple
  objects~in~$\cA$. Note that every object $A$ in $\cA_\mathrm{ss}$
  has unique decomposition $A=S_1^{a_1}\oplus\cdots\oplus S_n^{a_n}$
  where $a_1,\ldots,a_n \in \mathbb{N}_0$; we used here the assumption
  that $A$ has finite length to conclude that the cardinal numbers
  $a_i$ must be finite.  Since one has $\Hom_\cA(S_i,S_j)=0$ for $i
  \neq j$, it follows that there is an equivalence of abelian
  categories,
  \begin{displaymath}
    \cA_\mathrm{ss} 
    \simeq
    (\add\,S_1) \times \cdots \times (\add\,S_n)\;.
  \end{displaymath}
  Consider the ring \mbox{$D_i = \End_\cA(S_i)^{\mathrm{op}}$}.  As
  $S_i$ is simple, Schur's lemma gives that $D_i$ is a division ring.
  It easy to see that the functor \mbox{$\Hom_\cA(S_i,-) \colon \cA
    \to \Mod\,D_i$} induces an equivalence $\add\,S_i \simeq
  \proj\,D_i$.
  By \thmref{Gersten-transformation-1} the maps
  $\zeta_{\proj\,D_1},\ldots,\zeta_{\proj\,D_n}$ are isomorphisms, so
  it follows from the equivalence above, and the general observation
  in the beginning af the proof, that $\zeta_{\cA_\mathrm{ss}}$ is an
  isomorphism, as desired.
\end{proof}

Note that in this section, superscripts ``C'' (for classical), ``B''
(for Bass), and ``Q'' (for Quillen) have been used to distinguish between
various $\K$-groups. In the rest of the paper, $\K$-groups without
superscripts refer to Quillen's $\K$-groups.

\section{Coherent Pairs}
\label{sec:coherent}

We recall a few results and notions from the paper \cite{MR816889} by
Auslander and Reiten which are central in the proof of our main
\thmref{1}.  Throughout this section, $\cA$ denotes a skeletally small
additive category.

\begin{dfn}
  \label{dfn:weak}
  A \emph{pseudo} (or \emph{weak}) kernel of a morphism $g \colon A
  \to A'$ in $\cA$ is~a morphism $f \colon A'' \to A$ in $\cA$ such
  that $gf=0$, and which satisfies that every dia\-gram in $\cA$ as
  below can be completed (but not necessarily in a unique way).
  \begin{displaymath}
    \xymatrix{
      & B \ar[d]_-{h} \ar[dr]^-{0} \ar@{-->}[dl] & \\
      A'' \ar[r]_-{f} & A \ar[r]_-{g} & A'.
    }
  \end{displaymath}
  We say that $\cA$ \emph{has pseudo kernels} if every morphism in
  $\cA$ has a pseudo kernel.
\end{dfn}

\begin{obs}
  \label{obs:precover}
  Let $\cA$ be a full additive subcategory of an abelian category
  $\cM$. \new{An \emph{$\cA$-precover} of an object  $M \in \cM$ is a morphism $u \colon A \to M$ with $A \in \cA$ with the property that for every morphism $u' \colon A' \to M$ with $A' \in \cA$ there exists a (not necessarily unique) morphism $v \colon A' \to A$ such that $uv=u'$. Following \cite[def.~5.1.1]{rha} we
say that $\cA$ is precovering (or contravariantly finite) in $\cM$ if every object $M \in \cM$ has an $\cA$-precover. In this case, $\cA$ has pseudo kernels.} Indeed,
  if $i \colon K \to A$ is the kernel in $\cM$ of $g \colon A \to A'$
  in $\cA$, and if $f \colon A'' \to K$ is an $\cA$-precover of $K$,
  then $if \colon A'' \to A$ is a pseudo kernel of $g$.
\end{obs}

\begin{dfn}
  \label{res:coherent}
  Let $\cB$ be a full additive subcategory of $\cA$. Auslander and
  Reiten \cite{MR816889} call $(\cA,\cB)$ a \emph{coherent pair} if
  $\cA$ has pseudo kernels in the sense of \dfnref{weak}, and $\cB$ is
  precovering in $\cA$.
\end{dfn}

If $(\cA,\cB)$ is a coherent pair then also $\cB$ has pseudo kernels by
\cite[prop.~1.4(a)]{MR816889}.

\begin{dfn}
  \label{dfn:functor}
  Write $\Mod\,\cA$ for the abelian category of additive contravariant
  functors \mbox{$\cA \to \Ab$}, where $\Ab$ is the category of
  abelian groups. Denote by $\mod\,\cA$ the full subcategory of
  $\Mod\,\cA$ consisting of finitely presented functors.
\end{dfn}

\begin{ipg}
  \label{res:localization}
  If the category $\cA$ has pseudo kernels then $\mod\,\cA$ is
  abelian, and the inclusion functor $\mod\,\cA \to \Mod\,\cA$ is
  exact, see \cite[prop.~1.3]{MR816889}.

  If $(\cA,\cB)$ is a coherent pair, see \resref{coherent}, then the
  exact restriction $\Mod\,\cA \to \Mod\,\cB$ maps $\mod\,\cA$ to
  $\mod\,\cB$ by \cite[prop.~1.4(b)]{MR816889}. In this case, there
  are functors,
  \begin{equation}
    \label{eq:ir}
    \Ker r \stackrel{i}{\longrightarrow} 
    \mod\,\cA \stackrel{r}{\longrightarrow} \mod\,\cB\;,
  \end{equation}
  where $r$ is the restriction and $i$ the inclusion functor. The
  kernel of $r$, that is,
  \begin{displaymath}
    \Ker r = \{ F \in \mod\,\cA \ | \ F(B)=0 \text{ for all } B \in \cB \,\}\;,
  \end{displaymath}
  is a Serre subcategory of the abelian category $\mod\,\cA$. The
  quotient $(\mod\,\cA)/(\Ker r)$, in the sense of
  Gabriel~\cite{PGb62}, is equivalent to the category $\mod\,\cB$, and
  the canonical functor $\mod\,\cA \to (\mod\,\cA)/(\Ker r)$ may be
  identified with $r$. These assertions are proved in
  \cite[prop.~1.5]{MR816889}. Therefore \eqref{ir} induces by
  Quillen's localization theorem \mbox{\cite[\S5 thm.~5]{MR0338129}} a
  long exact sequence of $\K$-groups,
  \begin{equation}
    \label{eq:long-exact}
    \begin{gathered}
    \xymatrix@R=2ex{ 
      \cdots \ar[r] & \K_n(\Ker r) \ar[r]^-{\K_n(i)} &
      \K_n(\mod\,\cA) \ar[r]^-{\K_n(r)} & 
      \K_n(\mod\,\cB) \ar[r] & \cdots \\
      \cdots \ar[r] & \K_0(\Ker r) \ar[r]^-{\K_0(i)} &
      \K_0(\mod\,\cA) \ar[r]^-{\K_0(r)} & 
      \K_0(\mod\,\cB) \ar[r] & 0\,.
    }
    \end{gathered}
  \end{equation}
\end{ipg}

\section{Semilocal Rings}
\label{sec:semilocal}

A ring $A$ is semilocal if $A/\J(A)$ is semisimple.  Here $\J(A)$ is the
Jacobson radical of $A$. If $A$ is commutative then this definition is
equivalent to $A$ having only finitely many maximal ideals; see Lam
\cite[prop.~(20.2)]{MR1838439}.

\begin{lem}
  \label{lem:End-semilocal}
  Let $R$ be a commutative noetherian semilocal ring, and let \mbox{$M
    \neq 0$} be a finitely generated $R$-module. Then the ring
  $\End_R(M)$ is semilocal.
\end{lem}

\begin{proof}
  As $R$ is commutative and noetherian, $\End_R(M)$ is a module-finite
  $R$-algebra. Since $R$ is semilocal, the assertion now follows from
  \cite[prop.~(20.6)]{MR1838439}.
\end{proof}

\begin{ipg}
  \label{res:discussion}
  Denote by $A^*$ the group of units in a ring $A$, and let
  \mbox{$\vartheta_A \colon A^* \to \K_1^{\mathrm{C}}(A)$} be the
  com\-posite of the group homomorphisms,
  \begin{equation}
    \label{eq:vartheta}
    A^* \cong \GL_1(A)
    \hookrightarrow \GL(A) \twoheadrightarrow \GL(A)_{\mathrm{ab}} =
    \K_1^{\mathrm{C}}(A)\;.
  \end{equation}
  Some authors refer to $\vartheta_A$ as the Whitehead determinant.
  If $A$ is semilocal, then $\vartheta_A$ is surjective by Bass
  \cite[V\S9 thm.~(9.1)]{MR0249491}. As the group
  $\K_1^{\mathrm{C}}(A)$ is abelian one has $[A^*,A^*] \subseteq \Ker
  \vartheta_A$, and we write \mbox{$\theta_A \colon A^*_\mathrm{ab}
    \to \K_1^{\mathrm{C}}(A)$} for the induced homomorphism.

  Vaserstein \cite{MR0267009} showed that the inclusion $[A^*,A^*]
  \subseteq \Ker \vartheta_A$ is strict for the semi\-local ring
  $A=\M_2(\mathbb{F}_2)$ where $\mathbb{F}_2$ is the field with two
  elements. In \cite[thm.~3.6(a)]{MR0267009} it is shown that if $A$
  is semilocal, then $\Ker \vartheta_A$ is the subgroup of $A^*$
  generated by elements of the form $(1+ab)(1+ba)^{-1}$ where $a,b \in
  A$ and $1+ab \in A^*$.
\end{ipg}

If $A$ is semilocal, that is, $A/\J(A)$ is semisimple, then by the
Artin--Wedderburn theorem there is an isomorphism of rings,
\begin{displaymath}
  A/\J(A) \cong \M_{n_1}(D_1) \times \cdots \times \M_{n_t}(D_t)\;,
\end{displaymath}
where $D_1,\ldots,D_t$ are division rings, and $n_1,\ldots,n_t$ are
natural numbers all of which are uniquely determined by $A$.  The next
result is due to Vaserstein \cite[thm.~2]{MR2111217}.

\begin{thm}
  \label{thm:V2}
  Let $A$ be semilocal and write $A/\J(A) \cong \M_{n_1}(D_1) \times
  \cdots \times \M_{n_t}(D_t)$. If none of the $\M_{n_i}(D_i)$'s is
  $\M_2(\mathbb{F}_2)$, and at most one of the $\M_{n_i}(D_i)$'s is
  $\M_1(\mathbb{F}_2)=\mathbb{F}_2$ then one has $\Ker \vartheta_A =
  [A^*,A^*]$. In particular, $\vartheta_A$ induces an isomorphism,
  \begin{displaymath}
    \smash{\theta_A \colon
    A^*_{\mathrm{ab}} \stackrel{\cong}{\longrightarrow}
    \K_1^{\mathrm{C}}(A)\;.} \tag*{\qed}
  \end{displaymath}
\end{thm}

\begin{rmk}
  \label{rmk:char-neq-2}
  Note that if $A$ is a semilocal ring which is an algebra over a
  field $k$ with cha\-rac\-te\-ris\-tic \mbox{$\neq 2$}, then the
  hypothesis in \thmref{V2} is satisfied.
\end{rmk}

If $A$ is a commutative semilocal ring, then $\Ker \vartheta_A$ and
the commutator subgroup
\mbox{$[A^*,A^*]\mspace{-2mu}=\mspace{-2mu}\{1\}$} are identical,
i.e.~the surjective homomorphism \mbox{$\vartheta_A
  \mspace{-2mu}=\mspace{-2mu} \theta_A \colon \mspace{-2mu} A^* \to
  \K_1^{\mathrm{C}}(A)$} is an isomorphism.  Indeed, the determinant
homomorphisms \mbox{$\det_n \colon \GL_n(A) \to A^*$} induce a
homomorphism \mbox{$\det_A \colon \K_1^{\mathrm{C}}(A) \to A^*$} that
evidently satisfies \mbox{$\det_A\theta_A = 1_{A^*}$}. Since
$\theta_A$ is surjective, it follows that $\theta_A$ is an isomorphism
with $\theta_A^{-1} = \det_A$.

\begin{dfn}
  \label{dfn:det}
  Let $A$ be a ring for which the homomorphism \mbox{$\theta_A \colon
    A^*_{\mathrm{ab}} \to \K_1^{\mathrm{C}}(A)$} from
  \resref{discussion} is an isomorphism; for example, $A$ could be a
  commutative semilocal ring or a noncommutative semilocal ring
  satisfying the assumptions in \thmref{V2}.  The inverse
  \smash{$\theta_A^{-1}$} is denoted by $\det_A$, and we call it the
  \emph{generalized determinant}.
\end{dfn}

\begin{rmk}
  \label{rmk:cdot}
  Let $\xi$ be an $m \times n$ and let $\chi$ be an $n \times p$ matrix with
  entries in a ring $A$.  Denote by ``$\cdot$'' the product $\M_{m
    \times n}(A^\mathrm{op}) \times \M_{n \times p}(A^\mathrm{op}) \to
  \M_{m \times p}(A^\mathrm{op})$. Then
  \begin{displaymath}
    (\xi \cdot \chi)^T = \chi^T\xi^T\;,
  \end{displaymath}
  where $\chi^T\xi^T$ is computed using the product $\M_{p \times
    n}(A) \times \M_{n \times m}(A) \to \M_{p \times m}(A)$.  Thus,
  transposition \mbox{$(-)^T \colon \GL_n(A^\mathrm{op}) \to
    \GL_n(A)$} is an anti-isomorphism (this is also noted in
  \cite[V\S7]{MR0249491}), which induces an isomorphism \mbox{$(-)^T
    \colon \K_1^{\mathrm{C}}(A^\mathrm{op}) \to
    \K_1^{\mathrm{C}}(A)$}.
\end{rmk}

\begin{lem}
  \label{lem:det-T}
  Let $A$ be a ring for which the generalized determinant $\det_A =
  \theta_A^{-1}$ exists; cf.~\dfnref{det}.  For every invertible
  matrix $\xi$ with entries in $A$ one has an equality
  $\det_{A^\mathrm{op}}(\xi^T) = \det_A(\xi)$ in the abelian group
  $(A^\mathrm{op})^*_\mathrm{ab} = A^*_\mathrm{ab}$.
\end{lem}

\begin{proof}
  Clearly, there is a commutative diagram,
  \begin{displaymath}
    \xymatrix{
      A^*_\mathrm{ab} \ar@{=}[r] \ar[d]_-{\theta_A}^-{\cong}
      & (A^\mathrm{op})^*_\mathrm{ab} 
      \ar[d]^-{\theta_{A^\mathrm{op}}}_-{\cong} \\
      \K_1^{\mathrm{C}}(A) \ar[r]^-{\cong}_-{(-)^T} 
      & \K_1^{\mathrm{C}}(A^\mathrm{op})\,,
    }
  \end{displaymath}
  It follows that one has \mbox{$\theta_{A^\mathrm{op}}^{-1} \circ
    (-)^T = \theta_A^{-1}$}, that is, $\det_{A^\mathrm{op}} \circ\,
  (-)^T = \det_A$.
\end{proof}

\section{Some Useful Functors}
\label{sec:functors}

Throughout this section, $A$ is a ring and $M$ is a fixed left
$A$-module.  We denote by \mbox{$E=\End_A(M)$} the endomorphism ring
of $M$.  Note that \mbox{$M={}_{A,E}M$} has a natural
left-$A$-left-$E$--bimodule structure.

\begin{ipg}
  \label{res:projectivization}
  There is a pair of adjoint functors,
  \begin{displaymath}
    \xymatrix@C=15ex{
      \Mod\,A \ar@<0.7ex>[r]^-{\Hom_A(M,-)} & \Mod(E^\mathrm{op})\;.
      \ar@<0.7ex>[l]^-{-\otimes_EM}
    }
  \end{displaymath}
  It is easily seen that they restrict to a pair of quasi-inverse
  equivalences,
  \begin{displaymath}
    \xymatrix@C=15ex{
      \add_AM \ar@<1ex>[r]^-{\Hom_A(M,-)}_-{\simeq} & \proj(E^\mathrm{op})\;.
      \ar@<1ex>[l]^-{-\otimes_EM}
    }
  \end{displaymath}
  Auslander referred to this phenomenon as \emph{projectivization};
  see \cite[I\S2]{rta}.
\end{ipg}

Let \mbox{$F \in \Mod(\add_AM)$}, that is, \mbox{$F \colon \add_AM \to
  \Ab$} is a contravariant additive functor, see \dfnref{functor}. The
compatible $E$-module structure on the given $A$-module $M$ induces an
$E^\mathrm{op}$-module structure on the abelian group $FM$ which is
given by $z\alpha = (F\alpha)(z)$ for $\alpha \in E$ and $z \in FM$.

\begin{prp}
  \label{prp:e-f}
  There are quasi-inverse equivalences of abelian categories,
  \begin{displaymath}
    \xymatrix@C=15ex{
      \Mod(\add_AM) \ar@<1ex>[r]^-{e_M}_-{\simeq} & \Mod(E^\mathrm{op})\;,
      \ar@<1ex>[l]^-{f_M}
    }
  \end{displaymath}
  where $e_M$ (evaluation) and $f_M$ (functorfication) are defined as
  follows,
  \begin{align*}
    e_M(F) = FM \qquad \textnormal{and} \qquad
    f_M(Z) = Z \otimes_E\Hom_A(-,M)|_{\add_AM}\;,
  \end{align*}
  for $F$ in $\Mod(\add_AM)$ and $Z$ in $\Mod(E^\mathrm{op})$.  They
  restrict to quasi-inverse equivalences between categories of
  finitely presented objects,
  \begin{displaymath}
    \xymatrix@C=15ex{
      \mod(\add_AM) \ar@<1ex>[r]^-{e_M}_-{\simeq} & \mod(E^\mathrm{op})\;.
      \ar@<1ex>[l]^-{f_M}
    }
  \end{displaymath}
\end{prp}

\begin{proof}
  For $Z$ in $\Mod(E^\mathrm{op})$ the canonical isomorphism
  \begin{displaymath}
    \smash{Z \stackrel{\cong}{\longrightarrow} Z \otimes_E E = 
    Z \otimes_E\Hom_A(M,M) = e_Mf_M(Z)}
  \end{displaymath}
  is natural in $Z$. Thus, the functors $\mathrm{id}_{\Mod(E^\mathrm{op})}$ and
  $e_Mf_M$ are naturally isomorphic. For $F$ in $\Mod(\add_AM)$ there
  is a natural transformation,
  \begin{equation}
    \label{eq:fe}
    \smash{f_Me_M(F)=FM \otimes_E\Hom_A(-,M)|_{\add_AM} \stackrel{\delta}{\longrightarrow} F}\;;
  \end{equation}
  for $X$ in $\add_AM$ the homomorphism \mbox{$\delta_X \colon FM
    \otimes_E\Hom_A(X,M) \to FX$} is given by \mbox{$z \otimes \psi
    \mapsto (F\psi)(z)$}.  Note that $\delta_M$ is an isomorphism as
  it may be identified with the canonical isomorphism \smash{$FM
    \otimes_E {}_EE \stackrel{\cong}{\longrightarrow} FM$} in $\Ab$.
  As the functors in \eqref{fe} are additive, it follows that
  $\delta_X$ is an isomorphism for every \mbox{$X \in \add_AM$}, that
  is, $\delta$ is a natural isomorphism. Since \eqref{fe} is natural
  in $F$, the functors $f_Me_M$ and $\mathrm{id}_{\Mod(\add_AM)}$ are
  naturally isomorphic.

  It is straightforward to verify that the functors $e_M$ and $f_M$
  map finitely presented objects to finitely presented objects.
\end{proof}

\begin{obs}
  \label{obs:M-is-A}
  In the case $M=A$ one has \mbox{$E=\End_A(M)=A^\mathrm{op}$}, and
  therefore \prpref{e-f} yields an equivalence \mbox{$f_A \colon
    \mod\,A \to \mod(\proj\,A)$} given by
  \begin{displaymath}
    X \longmapsto X \otimes_{A^\mathrm{op}}\Hom_A(-,A)|_{\proj\,A}\;.
  \end{displaymath}
  It is easily seen that the functor $f_A$ is naturally isomorphic to
  the functor given by 
  \begin{displaymath}
    X \longmapsto \Hom_A(-,X)|_{\proj\,A}\;.
  \end{displaymath}
  We will usually identify $f_A$ with this functor.
\end{obs}

\begin{dfn}
  \label{dfn:2Y}
  The functor \mbox{$y_M \colon \add_AM \to \mod(\add_AM)$} which for
  $X \in \add_AM$ is given by $y_M(X) = \Hom_A(-,X)|_{\add_AM}$ is
  called the \emph{Yoneda functor}.
\end{dfn}

Let $\cA$ be a full additive subcategory of an abelian category $\cM$.
If $\cA$ is closed under extensions in $\cM$, then $\cA$ has a natural induced
exact structure. However, one can always equip $\cA$ with the
\emph{trivial exact structure}. In this structure, the ``exact
sequences'' (somtimes called \emph{conflations}) are only the split
exact ones.  When viewing $\cA$ as an exact category with the trivial
exact structure, we denote it $\cA_0$.

\begin{lem}
  \label{lem:Yoneda}
  Assume that $A$ is commutative and noetherian and let \mbox{$M \in
    \mod\,A$}. Set \mbox{$E=\End_A(M)$} and assume that
  $E^{\mathrm{op}}$ has finite global dimension. For the exact Yoneda
  functor \mbox{$y_M \colon (\add_AM)_0 \to \mod(\add_AM)$}, see
  \dfnref[]{2Y}, the homomorphisms $\K_n(y_M)$, where $n \geqslant 0$,
  and $\K_1^{\mathrm{B}}(y_M)$ are isomorphisms.
\end{lem}

\begin{proof}
  By application of $\K_n$ to the commutative diagram,
  \begin{displaymath}
    \xymatrix@C=12ex{ (\add_AM)_0 \ar[d]_-{y_M}
      \ar[r]^-{\Hom_A(M,-)}_-{\simeq} &
      \proj(E^{\mathrm{op}}) \ar[d]^-{\mathrm{inc}} \\
      \mod(\add_AM) \ar[r]_-{e_M}^-{\simeq} & \mod(E^{\mathrm{op}})\;,
    }
  \end{displaymath}
  it follows that $\K_n(y_M)$ is an isomorphism if and only if
  $\K_n(\mathrm{inc})$ is an isomorphism. The latter holds by
  Quillen's resolution theorem \cite[\S4 thm.~3]{MR0338129}, since
  $E^{\mathrm{op}}$ has finite global dimension.
  A similar argument shows that $\K_1^{\mathrm{B}}(y_M)$ is an
  isomorphism. This time one needs to apply Bass' resolution theorem; see
  \cite[VIII\S4 thm.~(4.6)]{MR0249491}.
\end{proof}

Since $\K_0$ may be identified with the Grothendieck group functor,
cf.~\resref{Quillen}, the following result is well-known. In any case,
it is straightforward to verify.

\begin{lem}
  \label{lem:intuitive}
  Assume that $\mod\,A$ is Krull--Schmidt. Let \mbox{$N = N_1^{n_1}
    \oplus \cdots \oplus N_s^{n_s}$} be a finitely generated
  $A$-module, where $N_1,\ldots,N_s$ are non-isomorphic indecomposable
  $A$-modules  and \mbox{$n_1,\ldots,n_s>0$}. The homomorphism of
  abelian groups,
  \begin{displaymath}
    \psi_N \colon
    \mathbb{Z}N_1 \oplus \cdots \oplus \mathbb{Z}N_s 
    \longrightarrow \K_0((\add_AN)_0)\;,
  \end{displaymath}
  given by $N_j \mapsto [N_j]$, is an isomorphism. \qed
\end{lem}

\section{The Abelian Category $\mathcal{Y}$}
\label{sec:Y}

By the assumptions in \stpref{assumptions}, the ground ring $R$ has a
dualizing module. It follows from Auslander and Buchweitz
\cite[thm.~A]{MAsROB89} that $\MCM\,R$ is precovering in $\mod\,R$.
Actually, in our case $\MCM\,R$ equals $\add_RM$ for some finitely
generated $R$-module $M$ (a representation generator), and it is
easily seen that every category of this form is precovering in
$\mod\,R$. By \obsref{precover} we have a coherent pair
$(\MCM\,R,\proj\,R)$, which by \resref{localization} yields a Gabriel
localization sequence,
\begin{equation}
  \label{eq:MCM-proj-ir}
  \mathcal{Y}=\Ker r \stackrel{i}{\longrightarrow} 
  \mod(\MCM\,R) \stackrel{r}{\longrightarrow} \mod(\proj\,R)\;.
\end{equation}
Here $r$ is the restriction functor, $\mathcal{Y}=\Ker r$, and $i$ is
the inclusion. Since an additive functor vanishes on $\proj\,R$ if and
only if it vanishes on $R$, one has
\begin{displaymath}
  \mathcal{Y} = \{ F \in \mod(\MCM\,R) \,|\, F(R)=0 \}\;.  
\end{displaymath}
The following two results about the abelian category $\mathcal{Y}$ are
due to Yoshino. The first result is \cite[(13.7.4)]{yos}; the
second is (proofs of) \cite[lem.~(4.12) and prop.~(4.13)]{yos}.

\begin{thm}
  \label{thm:Yos1}
  Every object in $\mathcal{Y}$ has finite length, i.e.~$\mathcal{Y}$ is a length category.  \qed
\end{thm}

\begin{thm}
  \label{thm:Yos2}
  Consider for \mbox{$1 \leqslant j \leqslant t$} the
  Auslander--Reiten sequence \eqref{AR} ending in $M_j$. The functor
  $F_j$, defined by the following exact sequence in $\mod(\MCM\,R)$,
  \begin{displaymath}
    0 \longrightarrow \Hom_R(-,\tau(M_j))
    \longrightarrow \Hom_R(-,X_j)
    \longrightarrow \Hom_R(-,M_j)
    \longrightarrow F_j
    \longrightarrow 0\;,
  \end{displaymath}
  is a simple object in $\mathcal{Y}$. Conversely, every simple
  functor in $\mathcal{Y}$ is naturally isomorphic to $F_j$ for some
  $1 \leqslant j \leqslant t$.  \qed
\end{thm}

\begin{prp}
  \label{prp:K0i-Y}
  Let \mbox{$i \colon \mathcal{Y} \to \mod(\MCM\,R)$} be the inclusion
  functor from \eqref{MCM-proj-ir} and let $\Upsilon \colon
  \mathbb{Z}^t \to \mathbb{Z}^{t+1}$ be the Auslander--Reiten
  homomorphism; see \dfnref{AR-matrix}. The homomorphisms $\K_0(i)$ and
  $\Upsilon$ are isomorphic.
\end{prp}

\begin{proof}
  We claim that the following diagram of abelian groups is commutative,
  \begin{displaymath}
    \xymatrix@R=1.5pc{
      \mathbb{Z}M_1 \oplus \cdots \oplus \mathbb{Z}M_t \ar[r]^-{\Upsilon} 
      \ar[dd]_-{\varphi}^-{\cong} & 
      \mathbb{Z}M_0 \oplus \mathbb{Z}M_1 \oplus \cdots \oplus \mathbb{Z}M_t 
      \ar[d]^-{\psi_M}_-{\cong} \\
      {} & \K_0((\MCM\,R)_0) \ar[d]^-{\K_0(y_M)}_-{\cong} \\
      \K_0(\mathcal{Y}) \ar[r]^-{\K_0(i)} & \K_0(\mod(\MCM\,R))\;.
    }
  \end{displaymath}
  The homomorphism $\varphi$ is defined by \mbox{$M_j \mapsto [F_j]$}
  where $F_j \in \mathcal{Y}$ is described in~\thmref[]{Yos2}. From
  \thmref[Theorems~]{Yos1} and \thmref[]{Yos2} and the proof of
  Rosenberg \cite[thm.~3.1.8(1)]{MR1282290} \new{(or the proof of \thmref{Gersten-transformation-2})}, it follows that $\varphi$
  is an isomorphism.  The module $M$ is a representation generator for
  $\MCM\,R$, see \stpref[]{assumptions}, and $\psi_M$ is the
  isomorphism given in \lemref{intuitive}.  Finally, $y_M$ is the
  Yoneda functor from \dfnref{2Y}.  By Leuschke \cite[thm.~6]{GJL07}
  the ring $E^\mathrm{op}$, where $E=\End_R(M)$, has finite global
  dimension, and thus \lemref{Yoneda} implies that $\K_0(y_M)$ is an
  isomorphism.

  From the definitions of the relevant homomorphims, it is straightforward to see that the diagram is commutative; indeed, both $\K_0(i)\varphi$ and $\K_0(y_M)\psi_M\Upsilon$ map a generator $M_j$ to the element $[F_j] \in \K_0(\mod(\MCM\,R))$.
\end{proof}

\section{Proof of the Main Theorem}
\label{sec:proof}

Throughout this section, we fix the notation in \stpref{assumptions}. Thus, $R$ is a commutative noetherian local Cohen--Macaulay ring satisfying conditions \stpref[]{assumptions}(1)--(3), $M$ is any representation generator of $\MCM\,R$, and $E$ is its endomorphism ring.

We shall frequently make use of the Gabriel localization sequence \eqref{MCM-proj-ir}, and $i$ and $r$ always denote the inclusion and the restriction functor in this sequence.

\begin{rmk}
  \label{rmk:C0-C}
  Let $\cC$ be an exact category.  As in the paragraph preceding
  \lemref[Lem\-ma~]{Yoneda}, we denote by $\cC_0$ the category $\cC$
  equipped with the trivial exact structure. Note that the identity
  functor \mbox{$\mathrm{id}_\cC \colon \cC_0 \to \cC$} is exact and
  the induced homomorphism \mbox{$\K_1^{\mathrm{B}}(\mathrm{id}_\cC)
    \colon \K_1^{\mathrm{B}}(\cC_0) \to \K_1^{\mathrm{B}}(\cC)$} is
  surjective, indeed, one has
  $\K_1^{\mathrm{B}}(\mathrm{id}_\cC)\big([C,\alpha]\big) =
  [C,\alpha]$.
\end{rmk}

\begin{lem}
  \label{lem:K1Br-surjective}
  Consider the restriction functor \mbox{$r \colon \mod(\MCM\,R) \to
    \mod(\proj\,R)$} and identity functor \mbox{$\mathrm{id}_{\MCM\,R}
    \colon (\MCM\,R)_0 \to \MCM\,R$}. The homomorphisms
  $\K_1^{\mathrm{B}}(r)$ and
  $\K_1^{\mathrm{B}}(\mathrm{id}_{\MCM\,R})$ are isomorphic, in
  particular, $\K_1^{\mathrm{B}}(r)$ is surjective by \rmkref{C0-C}.
\end{lem}

\begin{proof}
  Consider the commutative diagram of exact categories and exact
  functors,
  \begin{displaymath}
    \xymatrix@R=1.5pc{
      (\MCM\,R)_0 \ar[r]^-{\mathrm{id}_{\MCM\,R}} 
      \ar[dd]_-{y_M} & \MCM\,R 
      \ar[d]^-{j} \\
      {} & \mod\,R \ar[d]_-{\simeq}^-{f_R} \\
      \mod(\MCM\,R) \ar[r]^-{r} & \mod(\proj\,R)\;,
    }
  \end{displaymath}
  where $y_M$ is the Yoneda functor from \dfnref{2Y}, $j$ is the
  inclusion, and $f_R$ is the equivalence from \obsref{M-is-A}. We
  will prove the lemma by arguing that the vertical functors induce
  isomorphisms on the level of $\K_1^{\mathrm{B}}$.

  The ring $E^\mathrm{op}$ has finite global dimension by Leuschke
  \cite[thm.~6]{GJL07}, and hence \lemref{Yoneda} gives that that
  $\K_1^{\mathrm{B}}(y_M)$ is an isomorphism. Since $f_R$ is an
  equivalence, $\K_1^{\mathrm{B}}(f_R)$ is obviously an isomorphism.
  To argue that $\K_1^{\mathrm{B}}(j)$ is an isomorphism, we apply
  Bass' resolution theorem \cite[thm.~3.1.14]{MR1282290}. We must
  check that the subcategory $\MCM\,R$ of $\mod\,R$ satisfies
  conditions (1)--(3) in \emph{loc.\,cit}. Condition (1) follows as
  $\MCM\,R$ is precovering in $\mod\,R$.  As $R$ is Cohen--Macaulay,
  every module in $\mod\,R$ has a resolution of finite length by
  modules in $\MCM\,R$, see \cite[prop.~(1.4)]{yos}; thus condition
  (2) holds. Condition (3) requires that $\MCM\,R$ is closed under
  kernels of epimorphisms; this is well-known from
  e.g.~\cite[prop.~(1.3)]{yos}.
\end{proof}

Next we show some results on the Gersten--Sherman transformation; see~\secref[Sect.~]{K}.

\begin{lem}
  \label{lem:zeta-iso}
  $\zeta_\cC \colon \K_1^{\mathrm{B}}(\cC) \to \K_1(\cC)$ is an
  isomorphism for $\cC=\mod(\MCM\,R)$.
\end{lem}

\begin{proof}
  As $\zeta$ is a natural transformation, there is a commutative
  diagram,
  \begin{displaymath}
    \xymatrix@C10ex{
      \K_1^{\mathrm{B}}(\proj(E^{\mathrm{op}})) 
      \ar[d]_-{\zeta_{\proj(E^{\mathrm{op}})}}
      \ar[r]^-{\K_1^{\mathrm{B}}(\mathrm{inc})} &
      \K_1^{\mathrm{B}}(\mod(E^{\mathrm{op}})) 
      \ar[d]^-{\zeta_{\mod(E^{\mathrm{op}})}}
      \ar[r]^-{\K_1^{\mathrm{B}}(f_M)} & 
      \K_1^{\mathrm{B}}(\mod(\MCM\,R))\,\phantom{,} 
      \ar[d]^-{\zeta_{\mod(\MCM\,R)}} \\
      \K_1(\proj(E^{\mathrm{op}})) 
      \ar[r]^-{\K_1(\mathrm{inc})} &
      \K_1(\mod(E^{\mathrm{op}})) 
      \ar[r]^-{\K_1(f_M)} & 
      \K_1(\mod(\MCM\,R))\,,
    }
  \end{displaymath}
  where \mbox{$f_M \colon \mod(E^{\mathrm{op}}) \to \mod(\MCM\,R)$} is the
  equivalence from \prpref{e-f} and $\mathrm{inc}$ is the inclusion of
  $\proj(E^{\mathrm{op}})$ into $\mod(E^{\mathrm{op}})$.

  From Leuschke \cite[thm.~6]{GJL07}, the noetherian ring
  $E^{\mathrm{op}}$ has finite global dimension.  Hence Bass' and
  Quillen's resolution theorems, \cite[VIII\S4 thm.~(4.6)]{MR0249491}
  (see also Ro\-sen\-berg \cite[thm.~3.1.14]{MR1282290}) and \cite[\S4
  thm.~3]{MR0338129}, imply that $\K_1^{\mathrm{B}}(\mathrm{inc})$ and
  $\K_1(\mathrm{inc})$ are isomorphisms. Since $f_M$ is an
  equivalence, $\K_1^{\mathrm{B}}(f_M)$ and $\K_1(f_M)$ are
  isomorphisms as well. Consequently, $\zeta_{\mod(\MCM\,R)}$ is an
  isomorphism if and only if $\zeta_{\proj(E^{\mathrm{op}})}$ is an
  isomorphism, and the latter holds by
  \thmref{Gersten-transformation-1}.
\end{proof}

The goal is to compute Quillen's $\K$-group $\K_1(\mod\,R)$ for the ring $R$ in question. For our proof of \thmref{1}, it is crucial that this group can be naturally identified with Bass' $\K$-group $\K_1^{\mathrm{B}}(\mod\,R)$. To put \prpref{K1Q-is-K1B} in perspective, we remind the reader that the Gersten--Sherman transformation $\zeta_{\mod\,A}$ is not surjective for the ring $A=\mathbb{Z}C_2$; see~\cite[prop.~5.1]{MR0382398}.

\begin{prp}
  \label{prp:K1Q-is-K1B}
  If the Auslander--Reiten homomorphism from \dfnref{AR-matrix} is injective, then the following assertions hold:
  \begin{prt}
  \item The homomorphism $\zeta_{\mod\,R} \colon \K_1^{\mathrm{B}}(\mod\,R) \to \K_1(\mod\,R)$ is an isomorphism.
  \item There is an exact sequence,
     \begin{displaymath}  
       \xymatrix{\K_1^{\mathrm{B}}(\mathcal{Y}) \ar[r]^-{\K_1^{\mathrm{B}}(i)} &
       \K_1^{\mathrm{B}}(\mod(\MCM\,R)) \ar[r]^-{\K_1^{\mathrm{B}}(r)} &
       \K_1^{\mathrm{B}}(\mod(\proj\,R)) \ar[r] & 0}.
     \end{displaymath}  
  \end{prt}
\end{prp}

\begin{proof}
  The Gabriel localization sequence \eqref{MCM-proj-ir} induces by
  \resref{localization} a long exact sequence of Quillen
  $\K$-groups,
  \begin{displaymath}
    \xymatrix@C=4.3ex{ 
      \cdots \ar[r] & \K_1(\mathcal{Y}) \ar[r]^-{\K_1(i)} &
      \K_1(\mod(\MCM\,R)) \ar[r]^-{\K_1(r)} & 
      \K_1(\mod(\proj\,R)) \ar[r] & \K_0(\mathcal{Y}) 
      \ar[r]^-{\K_0(i)} & \cdots
    }
  \end{displaymath}
  By \prpref{K0i-Y}, we may identify $\K_0(i)$ with the
  Auslan\-der--Reiten homomorphism, which is assumed to be injective.
  Therefore, the bottom row in the following commutative diagram of
  abelian groups is exact,
  \begin{displaymath}
    \begin{gathered}
    \xymatrix@C=8ex{ 
      \K_1^{\mathrm{B}}(\mathcal{Y}) \ar[r]^-{\K_1^{\mathrm{B}}(i)}
      \ar[d]^-{\zeta_{\mathcal{Y}}}_-{\cong} &
      \K_1^{\mathrm{B}}(\mod(\MCM\,R)) \ar[r]^-{\K_1^{\mathrm{B}}(r)} 
      \ar[d]^-{\zeta_{\mod(\MCM\,R)}}_-{\cong} & 
      \K_1^{\mathrm{B}}(\mod(\proj\,R)) 
      \ar[d]^-{\zeta_{\mod(\proj\,R)}} \ar[r] & 0\phantom{.} \\
      \K_1(\mathcal{Y}) \ar[r]^-{\K_1(i)} &
      \K_1(\mod(\MCM\,R)) \ar[r]^-{\K_1(r)} & 
      \K_1(\mod(\proj\,R)) \ar[r] & 0.
    }
    \end{gathered}
  \end{displaymath}
  The vertical homomorphisms are given by the Gersten--Sherman
  transformation; see \secref{K}.  It follows from
  \thmref[Theorems~]{Yos1} and \thmref[]{Yos2} that $\mathcal{Y}$ is a
  length category with only finitely many simple objects; thus
  $\zeta_{\mathcal{Y}}$ is an isomorphism by
  \thmref{Gersten-transformation-2}. And $\zeta_{\mod(\MCM\,R)}$ is an
  isomorphism by \lemref{zeta-iso}. Since \mbox{$ri=0$}, it follows
  that $\K_1^{\mathrm{B}}(r)\K_1^{\mathrm{B}}(i)=0$ holds, and a diagram
  chase now shows that $\Im \K_1^{\mathrm{B}}(i) = \Ker
  \K_1^{\mathrm{B}}(r)$.  Furthermore $\K_1^{\mathrm{B}}(r)$ is
  surjective by \lemref{K1Br-surjective}. This proves part (b).

  The Five Lemma now implies that
  $\zeta_{\mod(\proj\,R)}$ is an isomorphism. Since the category
  $\mod(\proj\,R)$ is equivalent to $\mod\,R$, see \obsref{M-is-A}, it
  follows that $\zeta_{\mod\,R}$ is an isomorphism as well. This proves (a).
\end{proof}  

We will also need the following classical notion.

\begin{dfn}
  \label{dfn:cover}
  Let $\cM$ be an abelian category, and let $M$ be an object in $\cM$.
  A \emph{projective cover} of $M$ is an epimorphism $\varepsilon
  \colon P \twoheadrightarrow M$ in $\cM$, where $P$ is projective,
  such that every endomorphism $\alpha \colon P \to P$ satisfying
  $\varepsilon \alpha = \varepsilon$ is an automorphism.
\end{dfn}

\begin{lem}
  \label{lem:cover}
  Let there be given a commutative diagram,
  \begin{displaymath}
    \xymatrix{
      P \ar[d]_-{\alpha} \ar@{->>}[r]^-{\varepsilon} & M \ar[d]^-{\varphi} \\
      P \ar@{->>}[r]^-{\varepsilon} & M
    }
  \end{displaymath}
  in an abelian category $\cM$, where \mbox{$\varepsilon \colon P
    \twoheadrightarrow M$} is a projective cover of $M$. If $\varphi$
  is an automorphism, then $\alpha$ is an automorphism.
\end{lem}

\begin{proof}
  As $P$ is projective and $\varepsilon$ is an epimorphism, there
  exists \mbox{$\beta \colon P \to P$} such that \mbox{$\varepsilon
    \beta = \varphi^{-1} \varepsilon$}. By assumption one has
  \mbox{$\varepsilon \alpha = \varphi \varepsilon$}. Hence
  \mbox{$\varepsilon \alpha \beta = \varphi \varepsilon \beta =
    \varphi \varphi^{-1} \varepsilon = \varepsilon$}, and similarly,
  \mbox{$\varepsilon \beta \alpha = \varepsilon$}. As $\varepsilon$ is
  a projective cover, we conclude that $\alpha \beta$ and $\beta
  \alpha$ are automorphisms of $P$, and thus $\alpha$ must be an
  automorphism.
\end{proof}

The following lemma explains the point of the tilde construction
\resref{tilde}.


\begin{lem}
  \label{lem:eta}
  Consider the isomorphism \smash{\mbox{$\eta_{E^\mathrm{op}} \colon
    \K_1^{\mathrm{C}}(E^\mathrm{op}) \to
    \K_1^{\mathrm{B}}(\proj(E^\mathrm{op}))$}} in \resref{free}. Let $X \in \MCM\,R$ and $\alpha \in \Aut_R(X)$ be given, and $\tilde{\alpha}$ be the invertible matrix with entries in $E$ obtained by applying Construction \resref{tilde} to $\alpha$. There is an equality,
  \begin{displaymath}
    \eta_{E^\mathrm{op}}(\tilde{\alpha}^T) = \big[\Hom_R(M,X),\Hom_R(M,\alpha)\,\big]\;.
  \end{displaymath}
\end{lem}

\begin{proof}
  Write $(M,-)$ for $\Hom_R(M,-)$, and let \mbox{$\psi \colon X \oplus
    Y \stackrel{\cong}{\longrightarrow} M^q$} be as in Construction
  \resref{tilde}.  The $R$-module isomorphism $\psi$ induces an
  isomorphism of $E^\mathrm{op}$-modules,
  \begin{displaymath}
    \xymatrix@C=8ex{
      (M,X) \oplus (M,Y) = (M,X \oplus Y) \ar[r]^-{(M,\psi)}_-{\cong} & 
      (M,M^q) \cong E^q\,.
    }
  \end{displaymath}
  Consider the automorphism of the free $E^\mathrm{op}$-module $E^q$
  given by
  \begin{displaymath}
    (M,\psi) \big( (M,\alpha) \oplus 1_{(M,Y)} \big) (M,\psi)^{-1} 
    = (M,\psi(\alpha \oplus 1_Y)\psi^{-1}) 
    = (M,\tilde{\alpha})\;.
  \end{displaymath}
  We view elements in the $R$-module $M^q$ as columns and elements in
  $E^q$ as rows. The isomorphism \mbox{$E^q \cong (M,M^q)$} identifies
  a row vector \mbox{$\beta = (\beta_1,\ldots,\beta_q) \in E^q$} with
  the $R$-linear map \mbox{$\beta^T \colon M \to M^q$} whose
  coordinate functions are $\beta_1,\ldots,\beta_q$. The coordinate
  functions of \mbox{$(M,\tilde{\alpha})(\beta^T) =
    \tilde{\alpha}\circ \beta^T$} are the entries in the column
  $\tilde{\alpha}\beta^T$, where the matrix product used is
  \mbox{$\M_{q\times q}(E) \times \M_{q\times 1}(E) \to \M_{q\times
      1}(E)$}. Thus, the action of $(M,\tilde{\alpha})$ on a row
  \mbox{$\beta \in E^q$} is the row \mbox{$(\tilde{\alpha}\beta^T)^T
    \in E^q$}. In view of \rmkref{cdot} one has
  $(\tilde{\alpha}\beta^T)^T = \beta \cdot \tilde{\alpha}^T$, where
  ``$\cdot$'' is the product \mbox{$\M_{1\times q}(E^\mathrm{op})
    \times \M_{q\times q}(E^\mathrm{op}) \to \M_{1\times
      q}(E^\mathrm{op})$}. Consequently, over the ring
  $E^\mathrm{op}$, the automorphism $(M,\tilde{\alpha})$ of the
  $E^\mathrm{op}$-module $E^q$ acts on row vectors by multiplication
  with $\tilde{\alpha}^T$ from the right. These arguments show that 
$\eta_{E^\mathrm{op}}^{-1}$ applied to $[(M,X),(M,\alpha)]$ is $\tilde{\alpha}^T$; see \resref{free}.
\end{proof}

\begin{prp}
  \label{prp:sigma}
  Suppose, in addition to the blanket assumptions for this section, that $R$ is an algebra over its residue field $k$ and that \mbox{$\operatorname{char}(k) \neq 2$}. Then there is a group isomorphism,
  \begin{displaymath}
    \sigma \colon \Aut_R(M)_{\mathrm{ab}} \stackrel{\cong}{\longrightarrow} \K_1^{\mathrm{B}}(\mod(\MCM\,R))\;,
  \end{displaymath}
  given by
  \begin{displaymath}
    \alpha \longmapsto \big[\Hom_R(-,M)|_{\MCM\,R}\,,\, \Hom_R(-,\alpha)|_{\MCM\,R}\big]\;.
  \end{displaymath}
  Furthermore, there is an equality,
  \begin{displaymath}
    \sigma(\Xi) = \Im\K_1^{\mathrm{B}}(i)\;.
  \end{displaymath}
  Here $\Xi$ is the subgroup of $\Aut_R(M)_{\mathrm{ab}}$ given in \dfnref[]{Xi}, and $i \colon \mathcal{Y} \to \mod(\MCM\,R)$ is the inclusion functor from the Gabriel localization sequence \eqref{MCM-proj-ir}.
\end{prp}

\begin{proof}
  We define $\sigma$ to be the composite of the following isomorphisms,
  \begin{equation}
    \label{eq:sigma}
    \begin{gathered}
    \setlength\arraycolsep{0.1pt} 
    \begin{array}{ccl}
      \Aut_R(M)_{\mathrm{ab}} = E^*_{\mathrm{ab}} =
      (E^{\mathrm{op}})^*_{\mathrm{ab}}
      & \xymatrix@C=13ex{ \ar[r]_-{\cong}^-{\theta_{E^\mathrm{op}}} &} 
      & \K_1^{\mathrm{C}}(E^\mathrm{op}) \vspace*{-0.4ex} \\
      & \xymatrix@C=13ex{ \ar[r]_-{\cong}^-{\eta_{E^{\mathrm{op}}}} &} 
      & \K_1^{\mathrm{B}}(\proj(E^{\mathrm{op}})) \vspace*{-0.4ex} \\
      & \xymatrix@C=13ex{ \ar[r]_-{\cong}^-{\K_1^{\mathrm{B}}(\jmath)} &} 
      & \K_1^{\mathrm{B}}(\mod(E^{\mathrm{op}})) \vspace*{-0.4ex} \\
      & \xymatrix@C=13ex{ \ar[r]_-{\cong}^-{\K_1^{\mathrm{B}}(f_M)} &} 
      & \K_1^{\mathrm{B}}(\mod(\MCM\,R))\;.
    \end{array}
    \end{gathered}
  \end{equation}

  The ring $E$, and hence also its opposite ring $E^\mathrm{op}$, is
  semilocal by \lemref{End-semilocal}. By assumption, $R$ is a
  $k$-algebra, and hence so is $E^\mathrm{op}$.  Thus, in view of
  \rmkref{char-neq-2} and the assumption \mbox{$\operatorname{char}(k)
    \neq 2$}, we get the isomorphism \smash{$\theta_{E^\mathrm{op}}$}
  from \thmref{V2}. It maps \mbox{$\alpha \in
    \Aut_R(M)_{\mathrm{ab}}$} to the image of the
  \mbox{$1\!\times\!1$} matrix $(\alpha) \in \GL(E^{\mathrm{op}})$ in
  $\K_1^{\mathrm{C}}(E^\mathrm{op})$.

  The isomorphism $\eta_{E^{\mathrm{op}}}$ is described in
  \resref{free}; it maps \mbox{$\xi \in \GL_n(E^{\mathrm{op}})$} to
  the class \mbox{$[(E_E)^n,\xi] \in
    \K_1^{\mathrm{B}}(\proj(E^{\mathrm{op}}))$}.

  The third map in \eqref{sigma} is induced by the
  inclusion \mbox{$\jmath \colon \proj(E^{\mathrm{op}}) \to
    \mod(E^{\mathrm{op}})$}. By Leuschke \cite[thm.~6]{GJL07} the
  noetherian ring $E^{\mathrm{op}}$ has finite global dimension and
  hence Bass' resolution theorem \cite[VIII\S4 thm.~(4.6)]{MR0249491},
  or Rosenberg \cite[thm.~3.1.14]{MR1282290}, implies that
  $\K_1^{\mathrm{B}}(\jmath)$ is an isomorphism. It~maps an element
  $[P,\alpha] \in \K_1^{\mathrm{B}}(\proj(E^{\mathrm{op}}))$ to
  $[P,\alpha] \in \K_1^{\mathrm{B}}(\mod(E^{\mathrm{op}}))$.

  The fourth and last isomorphism $\K_1^{\mathrm{B}}(f_M)$ in \eqref{sigma} is induced by the equivalence $f_M \colon
  \mod(E^{\mathrm{op}}) \to \mod(\MCM\,R)$ from \prpref{e-f}.

  Thus, $\sigma$ is an isomorphism that
  maps an element \mbox{$\alpha \in \Aut_R(M)_{\mathrm{ab}}$} to the
  class
  \begin{displaymath}
    \big[E_E \otimes_E \Hom_R(-,M)|_{\MCM\,R}\,,\, (\alpha\,\cdot) \otimes_E \Hom_R(-,M)|_{\MCM\,R}\big]\;,
  \end{displaymath}
  which is evidently the same as the class
  \begin{displaymath}
    \big[\Hom_R(-,M)|_{\MCM\,R}\,,\, \Hom_R(-,\alpha)|_{\MCM\,R}\big]\;.
  \end{displaymath}
  
  It remains to show the equality $\sigma(\Xi) = \Im\K_1^{\mathrm{B}}(i)$. By the definition \eqref{sigma} of $\sigma$ this is tantamount to showing that $\K_1^{\mathrm{B}}(\jmath) \eta_{E^{\mathrm{op}}} \theta_{E^\mathrm{op}}(\Xi)= \K_1^{\mathrm{B}}(f_M)^{-1}(\Im\K_1^{\mathrm{B}}(i))$. As $e_M$ is a quasi-inverse of $f_M$, see \prpref{e-f}, we have
  \mbox{$\K_1^{\mathrm{B}}(f_M)^{-1} = \K_1^{\mathrm{B}}(e_M)$}, and hence we need to show the equality
\begin{equation}
  \label{eq:desired-identity}
  \K_1^{\mathrm{B}}(\jmath) \eta_{E^{\mathrm{op}}} \theta_{E^\mathrm{op}}(\Xi) \,=\, \K_1^{\mathrm{B}}(e_M)(\Im\K_1^{\mathrm{B}}(i))\;.
\end{equation}

By \dfnref{Xi}, the group $\Xi$ is generated by all elements of the form
  \begin{displaymath}
    \xi_{j,\alpha} := \textstyle{
    (\det_E\tilde{\alpha})
    (\det_E\tilde{\beta}_{j,\alpha})^{-1}
    (\det_E\tilde{\gamma}_{j,\alpha})
    } \in  E^*_{\mathrm{ab}}
  \end{displaymath}
  for $j \in \{1,\ldots,t\}$ and \mbox{$\alpha \in \Aut_R(M_j)$}; here $\beta_{j,\alpha} \!\in\!
  \Aut_R(X_j)$ and $\gamma_{j,\alpha} \!\in\! \Aut_R(\tau(M_j))$ are choices of automorphisms such that
  the diagram \eqref{j-alpha-diagram} is commutative. It follows from \lemref{det-T} that
  \begin{displaymath}
    \xi_{j,\alpha} =
    \textstyle{
    (\det_{E^\mathrm{op}}\tilde{\alpha}^T)
    (\det_{E^\mathrm{op}}\tilde{\beta}_{j,\alpha}^T)^{-1}
    (\det_{E^\mathrm{op}}\tilde{\gamma}_{j,\alpha}^T) \in (E^{\mathrm{op}})^*_{\mathrm{ab}}\;.
    }
  \end{displaymath}
  By \dfnref{det} the homomorphism $\det_{E^\mathrm{op}}$ is the inverse of $\theta_{E^\mathrm{op}}$, and consequently the group $\theta_{E^\mathrm{op}}(\Xi)$ is generated by the elements
  \begin{displaymath}
    \xi'_{j,\alpha} := \theta_{E^\mathrm{op}}(\xi_{j,\alpha}) =
    \tilde{\alpha}^T
    (\tilde{\beta}_{j,\alpha}^T)^{-1}
    \tilde{\gamma}_{j,\alpha}^T \in \K_1^{\mathrm{C}}(E^\mathrm{op})\;.
  \end{displaymath}
  Thus $\eta_{E^{\mathrm{op}}}\theta_{E^\mathrm{op}}(\Xi)$ is generated by the elements
$\xi''_{j,\alpha} := \eta_{E^{\mathrm{op}}}(\xi'_{j,\alpha}) \in \K_1^{\mathrm{B}}(\proj(E^{\mathrm{op}}))$, and it follows from \lemref{eta} that
  \begin{align*}
    \xi''_{j,\alpha} &= \big[\Hom_R(M,M_j),\Hom_R(M,\alpha)\,\big]
    - \big[\Hom_R(M,X_j),\Hom_R(M,\beta_{j,\alpha})\,\big] \\
    &\phantom{=} \quad
    + \big[\Hom_R(M,\tau(M_j)),\Hom_R(M,\gamma_{j,\alpha})\,\big]\;.
  \end{align*}
  Thus, the group $\K_1^{\mathrm{B}}(\jmath)\eta_{E^{\mathrm{op}}}\theta_{E^\mathrm{op}}(\Xi)$ on the left-hand side in 
  \eqref{desired-identity} is generated by the elements $\K_1^{\mathrm{B}}(\jmath)(\xi''_{j,\alpha})$. Note that $\K_1^{\mathrm{B}}(\jmath)(\xi''_{j,\alpha})$ is nothing but $\xi''_{j,\alpha}$ viewed as an element in $\K_1^{\mathrm{B}}(\mod(E^{\mathrm{op}}))$. We have reached the following conclusion:
  \begin{quote}
The group $\K_1^{\mathrm{B}}(\jmath)\eta_{E^{\mathrm{op}}}\theta_{E^\mathrm{op}}(\Xi)$ is generated by the elements $\xi''_{j,\alpha}$, where $j$ ranges over $\{1,\ldots,t\}$ and $\alpha$ over all automorphisms~of~$M_j$.
  \end{quote}

  To give a useful set of generators of the group $\K_1^{\mathrm{B}}(e_M)(\Im\K_1^{\mathrm{B}}(i))$ on the right-hand side in \eqref{desired-identity}, recall from
  \thmref[Theorems~]{Yos1} and \thmref[]{Yos2} that every element in 
  $\mathcal{Y}$ has finite length and that the simple objects in
  $\mathcal{Y}$ are, up to isomorphism, exactly the functors
  $F_1,\ldots,F_t$.  Thus, by 
  \cite[(proof of) thm.~3.1.8(2)]{MR1282290} the group
  $\K_1^{\mathrm{B}}(\mathcal{Y})$ is generated by all elements of the form
  $[F_j,\varphi]$, where $j \in \{1,\ldots,t\}$ and $\varphi$ is an
  auto\-morphism of $F_j$. It follows that the group $\Im\K_1^{\mathrm{B}}(i)$ is generated by the elements $\K_1^{\mathrm{B}}(i)([F_j,\varphi])$. Note that $\K_1^{\mathrm{B}}(i)([F_j,\varphi])$ is nothing but $[F_j,\varphi]$ viewed as an element in $\K_1^{\mathrm{B}}(\mod(\MCM\,R))$. By definition of the functor $e_M$, see \prpref[Prop.~]{e-f}, one has
  \begin{displaymath}
    \lambda_{j,\varphi} := \K_1^{\mathrm{B}}(e_M)([F_j,\varphi]) = [F_jM,\varphi_M]\;.
  \end{displaymath}
  We have reached the following conclusion:
  \begin{quote}
The group $\K_1^{\mathrm{B}}(e_M)(\Im\K_1^{\mathrm{B}}(i))$ is generated by the elements $\lambda_{j,\varphi}$, where $j$ ranges over $\{1,\ldots,t\}$ and $\varphi$ over all automorphisms of $F_j$.
  \end{quote}

  With the descriptions of the generators $\xi''_{j,\alpha}$ and $\lambda_{j,\varphi}$ at hand, we are now in a position to prove the identity \eqref{desired-identity}. 

Consider an arbitrary generator $\xi''_{j,\alpha}$ in the group $\K_1^{\mathrm{B}}(\jmath)\eta_{E^{\mathrm{op}}}\theta_{E^\mathrm{op}}(\Xi)$. Recall from \thmref{Yos2} that there is an exact sequence in
  $\mod(\MCM\,R)$,
  \begin{displaymath}
    0 \longrightarrow \Hom_R(-,\tau(M_j))
    \longrightarrow \Hom_R(-,X_j)
    \longrightarrow \Hom_R(-,M_j)
    \longrightarrow F_j
    \longrightarrow 0\;.
  \end{displaymath}
  Thus, the commutative diagram \eqref{j-alpha-diagram} in $\MCM\,R$
  induces a commutative diagram in $\mod(\MCM\,R)$ with exact row(s),
  \begin{displaymath}
    \begin{gathered}
      \xymatrix@R=5ex@C=4.5ex{ 0 \ar[r] & \Hom_R(-,\tau(M_j))
        \ar[d]_-{\cong}^-{\Hom_R(-,\gamma_{j,\alpha})} \ar[r] &
        \Hom_R(-,X_j) \ar[d]_-{\cong}^-{\Hom_R(-,\beta_{j,\alpha})}
        \ar[r] & \Hom_R(-,M_j) \ar[d]_-{\cong}^-{\Hom_R(-,\alpha)}
        \ar[r] &
        F_j \ar@{-->}[d]_-{\cong}^-{\varphi} \ar[r] & 0\;\phantom{,} \\
        0 \ar[r] & \Hom_R(-,\tau(M_j)) \ar[r] & \Hom_R(-,X_j) \ar[r] &
        \Hom_R(-,M_j) \ar[r] & F_j \ar[r] & 0\;,  }
    \end{gathered}
  \end{displaymath}  
  where $\varphi$ is the uniquely determined natural
  endotransformation of $F_j$ that makes this diagram commutative.
  Note that $\varphi$ is an automorphism by the Five Lemma, and thus $[F_j,\varphi]$ is a well-defined element in $\K_1^{\mathrm{B}}(\mod(\MCM\,R))$. The diagram above is an exact sequence in the loop category $\Omega(\mod(\MCM\,R))$, see \resref{loop-category} and \resref{K1B}, so in the group $\K_1^{\mathrm{B}}(\mod(\MCM\,R))$ there is an equality:
  \begin{align*}
    [F_j,\varphi] &= \big[\Hom_R(-,M_j),\Hom_R(-,\alpha)\,\big]
    - \big[\Hom_R(-,X_j),\Hom_R(-,\beta_{j,\alpha})\,\big] \\
    &\phantom{=} \quad
    + \big[\Hom_R(-,\tau(M_j)),\Hom_R(-,\gamma_{j,\alpha})\,\big]\;.
  \end{align*}
  Applying the homomorphism $\K_1^{\mathrm{B}}(e_M)$ to this equality, we get $\lambda_{j,\varphi}=\xi''_{j,\alpha}$. These arguments show that every generator $\xi''_{j,\alpha}$ has the form $\lambda_{j,\varphi}$ for some $\varphi$, and hence the inclusion "$\subseteq$" in \eqref{desired-identity} is established.

Conversely, consider an arbitrary generator  $\lambda_{j,\varphi}$ in the group $\K_1^{\mathrm{B}}(e_M)(\Im\K_1^{\mathrm{B}}(i))$. As the category $\MCM\,R$ is a Krull--Schmidt variety in the
  sense of Auslander \cite[II, \S2]{MR0349747}, it follows by
  \mbox{\cite[II, prop.~2.1(b,c)]{MR0349747}} and \mbox{\cite[I,
    prop.~4.7]{MR0349747}} that \mbox{$\Hom_R(-,M_j) \twoheadrightarrow F_j$}
  is a projective cover in $\mod(\MCM\,R)$ in the sense of
  \dfnref{cover}. In particular, $\varphi$ lifts to a natural
  transformation $\psi$ of $\Hom_R(-,M_j)$, which must be an
  automorphism by \lemref{cover}. Thus we have a commutative diagram
  in $\mod(\MCM\,R)$,
  \begin{displaymath}
    \xymatrix@R=5ex@C=4.5ex{
      \Hom_R(-,M_j) \ar@{-->}[d]_-{\psi}^-{\cong} \ar@{->>}[r]
      & F_j\;\phantom{.} \ar@<-3pt>[d]_-{\cong}^-{\varphi} \\
      \Hom_R(-,M_j) \ar@{->>}[r] & F_j\;.
    }
  \end{displaymath}
  As the Yoneda functor \mbox{$y_M \colon \MCM\,R \to
    \mod(\MCM\,R)$} is fully faithful, see \cite[lem. (4.3)]{yos},
  there exists a unique automorphism $\alpha$ of $M_j$ such that $\psi
  = \Hom_R(-,\alpha)$.  For this particular $\alpha$, the arguments
  above show that $\lambda_{j,\varphi}=\xi''_{j,\alpha}$. Thus every generator $\lambda_{j,\varphi}$ has the form $\xi''_{j,\alpha}$ for some $\alpha$, and hence the inclusion "$\supseteq$" in \eqref{desired-identity} holds.
\end{proof}

\begin{obs}
  \label{obs:rho}
  For any commutative noetherian local ring $R$, there is an
  iso\-morphism \smash{\mbox{$\rho_R \colon R^*
      \stackrel{\cong}{\longrightarrow} \K_1^{\mathrm{B}}(\proj\,R)$}}
  given by the composite of
  \begin{displaymath}
    \smash{\xymatrix{
      R^* \ar[r]_-{\cong}^-{\theta_R} & \K_1^{\mathrm{C}}(R)
      \ar[r]^-{\eta_R}_-{\cong} & \K_1^{\mathrm{B}}(\proj\,R)\;.
    }}
  \end{displaymath}
  The first map is described in \resref{discussion}; it is an
  isomorphism by Srinivas \cite[exa.~(1.6)]{MR1382659}. The second
  isomorphism is discussed in \resref{free}. Thus, $\rho_R$ maps $r
  \in R^*$ to $[R,r1_R]$.
\end{obs}

We are finally in a position to prove the main result.

\begin{proof}[Proof of \thmref{1}]
  By \prpref{K1Q-is-K1B} we can identify $\K_1(\mod\,R)$ with the group $\K_1^{\mathrm{B}}(\mod\,R)$. Recall that $i$ and $r$ denote the inclusion and restriction functors from the localization sequence \eqref{MCM-proj-ir}.
  By the relations that define $\K_1^{\mathrm{B}}(\mod\,R)$, see
  \resref{K1B}, there is a homomorphism \mbox{$\pi_0 \colon \Aut_R(M)
    \to \K_1^{\mathrm{B}}(\mod\,R)$} given by \mbox{$\alpha \mapsto
    [M,\alpha]$}. Since $\K_1^{\mathrm{B}}(\mod\,R)$ is abelian,
  $\pi_0$ induces a homomorphism $\pi$, which is displayed as the upper horizontal map in the following diagram,
  \begin{equation}
    \label{eq:pi-sigma}
    \begin{gathered}
    \xymatrix@C=7ex{ 
      \Aut_R(M)_{\mathrm{ab}} \ar[d]_-{\sigma}^-{\cong} \ar[r]^-{\pi} &
      \K_1^{\mathrm{B}}(\mod\,R) \ar[d]^-{\K_1^{\mathrm{B}}(f_R)}_-{\cong} \\
      \K_1^{\mathrm{B}}(\mod(\MCM\,R)) \ar[r]^-{\K_1^{\mathrm{B}}(r)} & 
      \K_1^{\mathrm{B}}(\mod(\proj\,R))\;.
    }
    \end{gathered}
  \end{equation}
  Here $\sigma$ is the isomorphism from \prpref{sigma}, and the isomorphism $\K_1^{\mathrm{B}}(f_R)$ is induced by the equivalence $f_R$ from \obsref{M-is-A}. The diagram \eqref{pi-sigma} is commutative, indeed, $\K_1^{\mathrm{B}}(r)\sigma$ and $\K_1^{\mathrm{B}}(f_R)\pi$ both map \mbox{$\alpha \in \Aut_R(M)_{\mathrm{ab}}$} to the class
  \begin{displaymath}
    \big[\Hom_R(-,M)|_{\proj\,R}\,,\, \Hom_R(-,\alpha)|_{\proj\,R}\big]\;.
  \end{displaymath}
  By \lemref{K1Br-surjective} the homomorphism $\K_1^{\mathrm{B}}(r)$ is surjective, and hence so is $\pi$.  Exactness of the sequence in \prpref{K1Q-is-K1B}(b) and commutativity of the diagram
  \eqref{pi-sigma} show that \mbox{$\Ker\pi =\sigma^{-1}(\Im\K_1^{\mathrm{B}}(i))$}. Therefore \prpref{sigma} implies that there is an equality $\Ker \pi = \Xi$, and it follows that $\pi$ induces an isomorphism,
  \begin{displaymath}
    \widehat{\pi} \colon \Aut_R(M)_{\mathrm{ab}}/\Xi
    \stackrel{\cong}{\longrightarrow} \K_1^{\mathrm{B}}(\mod\,R)\;.
  \end{displaymath}
  This proves the first assertion in \thmref{1}. 

  To prove the second assertion, let $\mathrm{inc} \colon \proj\,R \to \mod\,R$ denote the inclusion functor.
Note that the Gersten--Sherman transformation identifies the homomorphisms $\K_1(\mathrm{inc})$ and
  $\K_1^{\mathrm{B}}(\mathrm{inc})$; indeed $\zeta_{\proj\,R}$ is an isomorphism by \thmref{Gersten-transformation-1} and $\zeta_{\mod\,R}$ is an isomorphism by \prpref{K1Q-is-K1B}(a). Thus, we
  must show that $\K_1^{\mathrm{B}}(\mathrm{inc})$ can be identified
  with the homomorphism \mbox{$\lambda \colon R^* \to
    \Aut_R(M)_{\mathrm{ab}}/\Xi$} given by \mbox{$r \mapsto r1_R
    \oplus 1_{M'}$} (recall that we have written $M = R \oplus M'$).
To this end, consider the isomorphism
  \smash{\mbox{$\rho_R \colon R^* \to \K_1^{\mathrm{B}}(\proj\,R)$}}
  from \obsref{rho} given by \mbox{$r \mapsto [R,r1_R]$}. The fact
  that $\K_1^{\mathrm{B}}(\mathrm{inc})$ and $\lambda$ are isomorphic
  maps now follows from the diagram,
  \begin{equation*}
    \xymatrix@C=3pc{ 
      R^* \ar[d]^-{\cong}_-{\rho_R} \ar[r]^-{\lambda} & 
      \Aut_R(M)_{\mathrm{ab}}/\Xi \ar[d]^-{\widehat{\pi}}_-{\cong} \\
      \K_1^{\mathrm{B}}(\proj\,R) \ar[r]^-{\K_1^{\mathrm{B}}(\mathrm{inc})} & 
      \K_1^{\mathrm{B}}(\mod\,R)\;,
    }
  \end{equation*}
  which is commutative.  Indeed, for $r \in R^*$ one has
  \begin{displaymath}
    (\widehat{\pi}\lambda)(r) = [M,r1_R \oplus 1_{M'}] =
    [R,r1_R]+[M',1_{M'}] = [R,r1_R] =
    (\K_1^{\mathrm{B}}(\mathrm{inc})\rho_R)(r)\;,
  \end{displaymath}
  where the penultimate equality is by \resref{neutral}.
\end{proof}

\section{Abelianization of Automorphism Groups}
\label{sec:Aut-ab}

To apply \thmref{1}, one must compute $\Aut_R(M)_\mathrm{ab}$,
i.e.~the abelianization of the automorphism group of the
representation generator $M$.  In \prpref{Aut-ab-exa} we compute
$\Aut_R(M)_\mathrm{ab}$ for the $R$-module \mbox{$M=R \oplus \m$},
which is a representation generator for $\MCM\,R$ if $\m$ happens to
be the only non-free indecomposable maximal Cohen--Macaulay module
over $R$. Specific examples of rings for which this is the case will be studied in
\secref{examples}.  Throughout this section, $A$ denotes any ring.

\begin{dfn}
  \label{dfn:row-operations}
  Let $N_1,\ldots,N_s$ be $A$-modules, and set
  \mbox{$N=N_1\oplus\cdots\oplus N_s$}. We view elements in $N$ as
  column vectors.

  For \mbox{$\varphi \in \Aut_A(N_i)$} we denote by
  $d_i(\varphi)$ the automorphism of $N$ which has as its
  diagonal
  \smash{$1_{N_1},\ldots,1_{N_{i-1}},\varphi,1_{N_{i+1}},\ldots,1_{N_s}$}
  and $0$ in all other entries.

  For \mbox{$i \neq j$} and \mbox{$\mu \in \Hom_A(N_j,N_i)$} we
  denote by $e_{ij}(\mu)$ the automorphism of $N$ with diagonal
  \mbox{$1_{N_1},\ldots,1_{N_s}$}, and whose only non-trivial
  off-diagonal entry is $\mu$ in position $(i,j)$.
\end{dfn}

\begin{lem}
  \label{lem:e-commutator}
  Let \mbox{$N_1,\ldots,N_s$} be $A$-modules and set
  \mbox{$N=N_1\oplus\cdots\oplus N_s$}. If~\mbox{$\,2 \in A$} is a unit,
  if \mbox{$i \neq j$}, and if \mbox{$\mu \in \Hom_A(N_j,N_i)$} then
  $e_{ij}(\mu)$ is a commutator in $\Aut_A(N)$.
\end{lem}

\begin{proof}
  The commutator of $\varphi$ and $\psi$ in $\Aut_A(N)$ is
  $[\varphi,\psi]=\varphi\psi\varphi^{-1}\psi^{-1}$. It is easily
  verified that
  $e_{ij}(\mu)=[e_{ij}(\frac{\mu}{2}),d_j(-1_{N_j})]$ if $i
  \neq j$.
\end{proof}

The idea in the proof above is certainly not new. It appears, for
example, already in Litoff \cite[proof of thm.~2]{MR0068541} in the
case $s=2$. \new{Of course, if $s \geqslant 3$ then $e_{ij}(\mu)$ is a commutator even without the assumption that $2$ is a unit; see e.g.~\cite[lem.~2.1.2(c)]{MR1282290}.}

\begin{lem}
  \label{lem:factors}
  Let $X$ and $Y$ be non-isomorphic $A$-modules with local
  endomorphism rings.  Let \mbox{$\varphi, \psi \in \End_A(X)$} and
  assume that $\psi$ factors through $Y$.  Then one has \mbox{$\psi
    \notin \Aut_A(X)$}. Futhermore, \mbox{$\varphi \in \Aut_A(X)$} if
  and only if \mbox{$\varphi + \psi \in \Aut_A(X)$}.
\end{lem}

\begin{proof}
  Write \mbox{$\psi=\psi''\psi'$} with \mbox{$\psi'\colon X \to Y$}
  and \mbox{$\psi'' \colon Y \to X$}. If $\psi$ is an automorphism,
  then $\psi''$ is a split epimorphism and hence an isomorphism as $Y$
  is indecomposable. This contradicts the assumption that $X$ and $Y$
  are not isomorphic. The second assertion now follows as $\Aut_A(X)$
  is the set of units in the local ring $\End_A(X)$.
\end{proof}

\begin{prp}
  \label{prp:End}
  Let $N_1,\ldots,N_s$ be pairwise non-isomorphic $A$-modules with
  local endomorphism rings.  An endomorphism
  \begin{displaymath}
    \alpha = (\alpha_{ij}) \in
    \End_A(N_1\oplus\cdots\oplus N_s)
    \quad \text{with} \quad
    \alpha_{ij} \in
    \Hom_A(N_j,N_i)
  \end{displaymath}
  is an automorphism if and only if
  $\alpha_{11},\alpha_{22},\ldots,\alpha_{ss}$ are automorphisms.

  Furthermore, every $\alpha$ in $\Aut_A(N)$ can be written as a
  product of automorphisms of the form $d_i(\cdot)$ and
  $e_{ij}(\cdot)$, cf.~\dfnref{row-operations}.
\end{prp}

\begin{proof}
  ``Only if'': Assume that \mbox{$\alpha = (\alpha_{ij})$} is an
  automorphism with inverse \mbox{$\beta = (\beta_{ij})$} and let $i
  =1,\ldots,s$ be given. In the local ring $\End_A(N_i)$ one
  has \smash{\mbox{$1_{N_i}=\sum_{j=1}^s\alpha_{ij}\beta_{ji}$}}, and
  hence one of the terms $\alpha_{ij}\beta_{ji}$ must be an
  automorphism. As $\alpha_{ij}\beta_{ji}$ is not an automorphism for
  $j \neq i$, see \lemref{factors}, it follows that
  $\alpha_{ii}\beta_{ii}$ is an automorphism. \new{In particular, $\alpha_{ii}$ has a right inverse and $\beta_{ii}$ has a left inverse, and since the ring $\End_A(N_i)$ is local this means that $\alpha_{ii}$ and $\beta_{ii}$ are both automorphisms.} 

  ``If'': By induction on \mbox{$s\geqslant 1$}. The assertion is trivial for
  \mbox{$s=1$}. Now let \mbox{$s>1$}. Assume that
  $\alpha_{11},\alpha_{22},\ldots,\alpha_{ss}$ are automorphisms.
  Recall the notation from \dfnref[]{row-operations}.  By composing
  $\alpha$ with $e_{s1}(-\alpha_{s1}\alpha_{11}^{-1}) \cdots
  e_{31}(-\alpha_{31}\alpha_{11}^{-1})
  e_{21}(-\alpha_{21}\alpha_{11}^{-1})$ from the left and with
  $e_{12}(-\alpha_{11}^{-1}\alpha_{12})
  e_{13}(-\alpha_{11}^{-1}\alpha_{13}) \cdots
  e_{1s}(-\alpha_{11}^{-1}\alpha_{1s})$ from the right, one gets an
  endomorphism of the form
  \begin{displaymath}
    \alpha' = 
    \left(\!
      \begin{array}{c|c}
        \alpha_{11} & 0 \\
        \hline
        0 & \beta
      \end{array}
      \!\right)
      = d_1(\alpha_{11})
    \left(\!
      \begin{array}{c|c}
        1_{N_1} & 0 \\
        \hline
        0 & \beta
      \end{array}
      \!\right),
  \end{displaymath}
  where \mbox{$\beta \in \End_A(N_2 \oplus \cdots \oplus N_s)$} is an
  \mbox{$(s-1)\times(s-1)$} matrix with diagonal entries given by
  \mbox{$\alpha_{jj}-\alpha_{j1}\alpha_{11}^{-1}\alpha_{1j}$} for
  \mbox{$j=2,\ldots,s$}. By applying \lemref{factors} to the
  si\-tu\-a\-tion
  \mbox{$\varphi=\alpha_{jj}-\alpha_{j1}\alpha_{11}^{-1}\alpha_{1j}$}
  and \mbox{$\psi=\alpha_{j1}\alpha_{11}^{-1}\alpha_{1j}$}, it follows
  that the diagonal entries in $\beta$ are all automorphisms. By the
  induction hypothesis, $\beta$ is now an automorphism and can be
  written as a product of automorphisms of the form
  \smash{$d_i(\cdot)$} and \smash{$e_{ij}(\cdot)$}. Consequently, the
  same is true for \smash{$\alpha'$}, and hence also for $\alpha$.
\end{proof}

\begin{cor}
  \label{cor:Aut-ab}
  Assume that \mbox{$2 \in A$} is a unit and let $N_1,\ldots,N_s$ be
  pairwise non-isomorphic $A$-modules with local endomorphism rings.
  The homomorphism,
  \begin{displaymath}
    \Delta \colon \Aut_A(N_1) \times \cdots \times \Aut_A(N_s) 
    \longrightarrow \Aut_A(N_1 \oplus \cdots \oplus N_s)\;,
  \end{displaymath}
  given by $\Delta(\varphi_1,\ldots,\varphi_s) =
  d_1(\varphi_1)\cdots d_s(\varphi_s)$, induces a surjective homomorphism,
  \begin{displaymath}
    \Delta_\mathrm{ab} \colon \Aut_A(N_1)_\mathrm{ab} \oplus \cdots 
    \oplus \Aut_A(N_s)_\mathrm{ab}
    \longrightarrow \Aut_A(N_1 \oplus \cdots \oplus N_s)_\mathrm{ab}\;.
  \end{displaymath}
\end{cor}

\begin{proof}
  By \prpref{End} every element in $\Aut_A(N_1 \oplus \cdots \oplus
  N_s)$ is a product of automorphisms of the form $d_i(\cdot)$ and
  $e_{ij}(\cdot)$. As \mbox{$2 \in A$} is a unit,
  \lemref{e-commutator} yields that every element of the form
  $e_{ij}(\cdot)$ is a commutator; thus in $\Aut_A(N_1 \oplus \cdots
  \oplus N_s)_\mathrm{ab}$ every element is a product of elements of
  the form $d_i(\cdot)$, so $\Delta_\mathrm{ab}$ is surjective.
\end{proof}

\new{As noted above, \lemref{e-commutator}, and consequently also \corref{Aut-ab}, holds without the assumption that \mbox{$2 \in A$} is a unit provided that $s \geqslant 3$.}

In the following, we write $[\,\cdot\,]_\m \colon R \twoheadrightarrow R/\m=k$ for the quotient homomorphism.

\begin{prp}
  \label{prp:Aut-ab-exa}
  Let $(R,\m,k)$ be any commutative local ring such that \mbox{$2 \in
    R$} is a unit. Assume that $\m$ is not isomorphic to $R$ and that
  the endomorphism ring $\End_R(\m)$ is commutative and local. There is an
  isomorphism of abelian groups,
  \begin{displaymath}
    \delta \colon \Aut_R(R \oplus \m)_\mathrm{ab}
    \stackrel{\cong}{\longrightarrow}
    k^* \oplus \Aut_R(\m)\;,
  \end{displaymath}
  given by
  \begin{displaymath}
    \begin{pmatrix}
      \alpha_{11} & \alpha_{22} \\
      \alpha_{21} & \alpha_{22}
    \end{pmatrix}
    \longmapsto
    \big([\alpha_{11}(1)]_\m,\alpha_{11}\alpha_{22}-\alpha_{21}\alpha_{12}\big)\;.
   \end{displaymath}
\end{prp}

\begin{proof}
  First note that the image of any homomorphism \mbox{$\alpha \colon
    \m \to R$} is contained~in~$\m$. Indeed if \mbox{$\Im \alpha
    \nsubseteq \m$}, then \mbox{$u=\alpha(a)$} is a unit for some
  \mbox{$a \in \m$}, and thus \mbox{$\alpha(u^{-1}a) = 1$}. It follows
  that $\alpha$ is surjective, and hence a split epimorphism as $R$ is
  free.  Since $\m$ is indecomposable, $\alpha$ must be an
  isomorphism, which is a contradiction.

  Therefore, given an endomorphism,
  \begin{displaymath}
    \begin{pmatrix}
      \alpha_{11} & \alpha_{12} \\
      \alpha_{21} & \alpha_{22}
    \end{pmatrix}
    \in \End_R(R \oplus \m) = 
    \begin{pmatrix}
      \Hom_R(R,R) & \Hom_R(\m,R) \\
      \Hom_R(R,\m) & \Hom_R(\m,\m) \\
    \end{pmatrix},
  \end{displaymath} 
  we may by (co)restriction view the entries $\alpha_{ij}$ as elements
  in the endomorphism ring $\End_R(\m)$. As this ring is assumed to be
  commutative, the determinant map
  \begin{displaymath}
    \End_R(R \oplus \m) \longrightarrow \End_R(\m)
    \qquad \text{given by} \qquad
    (\alpha_{ij}) \longmapsto
    \alpha_{11}\alpha_{22}-\alpha_{21}\alpha_{12}
  \end{displaymath}
  preserves multiplication.  If \mbox{$(\alpha_{ij}) \in \Aut_R(R
    \oplus \m)$}, then \prpref{End} implies that \mbox{$\alpha_{11}
    \in \Aut_R(R)$} and \mbox{$\alpha_{22} \in \Aut_R(\m)$}, and thus
  \mbox{$\alpha_{11}\alpha_{22} \in \Aut_R(\m)$}.  By applying
  \lemref{factors} to
  \mbox{$\varphi=\alpha_{11}\alpha_{22}-\alpha_{21}\alpha_{12}$} and
  \mbox{$\psi=\alpha_{21}\alpha_{12}$} we get \mbox{$\varphi \in
    \Aut_R(\m)$}, and hence the determinant map is a group
  homomorphism $\Aut_R(R \oplus \m) \to \Aut_R(\m)$.

  The map \mbox{$\Aut_R(R \oplus \m) \to k^*$} defined by
  \mbox{$(\alpha_{ij}) \mapsto [\alpha_{11}(1)]_\m$} is also a group
  homomorphism. Indeed, entry $(1,1)$ in the product
  $(\alpha_{ij})(\beta_{ij})$ is
  \mbox{$\alpha_{11}\beta_{11}+\alpha_{12}\beta_{21}$}. Here
  $\alpha_{12}$ is a homomorphism \mbox{$\m \to R$}, and hence
  $\alpha_{12}\beta_{21}(1) \in \m$ by the arguments in the beginning
  of the proof. Consequently one has
  \begin{displaymath}
    [(\alpha_{11}\beta_{11}+\alpha_{12}\beta_{21})(1)]_\m =
    [(\alpha_{11}\beta_{11})(1)]_\m = [\alpha_{11}(1)\beta_{11}(1)]_\m =
    [\alpha_{11}(1)]_\m[\beta_{11}(1)]_\m\;.
  \end{displaymath}

  These arguments and the fact that the groups $k^*$ and $\Aut_R(\m)$
  are abelian show that the map $\delta$ described in the proposition
  is a well-defined group homomorphism. Evidently, $\delta$ is
  surjective; indeed, for $[r]_\m \in k^*$ and $\varphi \in
  \Aut_R(\m)$ one has
  \begin{displaymath}
    \delta 
    \begin{pmatrix}
      r1_R & 0 \\
      0 & r^{-1}\varphi
    \end{pmatrix}
    = ([r]_\m,\varphi)\;.
  \end{displaymath}
  To show that $\delta$ is injective, assume that $\alpha \in \Aut_R(R
  \oplus \m)_\mathrm{ab}$ with $\delta(\alpha) = ([1]_\m,1_\m)$.  By
  \corref{Aut-ab} we can assume that $\alpha = (\alpha_{ij})$ is a
  diagonal matrix. We write $\alpha_{11} = r1_R$ for some unit $r \in
  R$. Since one has $\delta(\alpha) = ([r]_\m,r\alpha_{22})$ we
  conclude that $r \in 1+\m$ and $\alpha_{22} = r^{-1}1_\m$, that is,
  $\alpha$ has the form
  \begin{displaymath}
    \alpha = 
    \begin{pmatrix}
      r1_R & 0 \\
      0 & r^{-1}1_{\m}
    \end{pmatrix} 
    \qquad \text{with} \qquad r \in 1+\m\;.
  \end{displaymath}
  Thus, proving injectivity of $\delta$ amounts to showing that every
  automorphism $\alpha$ of the form above belongs to the commutator
  subgroup of $\Aut_R(R \oplus \m)$. As \mbox{$r-1 \in \m$} the map
  $(r-1)1_R$ gives a homomorphism \mbox{$R \to \m$}.  Since
  \mbox{$r(r^{-1}-1) = 1-r \in \m$} and $r \notin \m$, it follows that
  \mbox{$r^{-1}-1 \in \m$}.  Thus $(r^{-1}-1)1_R$ gives another
  homomorphism \mbox{$R \to \m$}.  If \mbox{$\iota \colon \m
    \hookrightarrow R$} denotes the inclusion, then one has\footnote{\
    The identity comes from the standard proof of Whitehead's lemma;
    see e.g.~\lemcite[(1.4)]{MR1382659}.}
  \begin{align*}
    \begin{pmatrix}
      r1_R & 0 \\
      0 & r^{-1}1_{\m}
    \end{pmatrix}
    &=
    \begin{pmatrix}
      1_R & 0 \\
      (r^{-1}-1)1_R & 1_\m
    \end{pmatrix}
    \!
    \begin{pmatrix}
      1_R & \iota \\
      0 & 1_\m
    \end{pmatrix}
    \!
    \begin{pmatrix}
      1_R & 0 \\
      (r-1)1_R & 1_\m
    \end{pmatrix}
    \!
    \begin{pmatrix}
      1_R & -r^{-1}\iota \\
      0 & 1_\m
    \end{pmatrix}.
  \end{align*}
  The right-hand of this equality is a product of matrices of the form
  $e_{ij}(\cdot)$, and since $2 \in R$ is a unit the desired
  conclusion now follows from \lemref{e-commutator}.
\end{proof}

\section{Examples} \label{sec:examples}

We begin with a trivial example.

\begin{exa}
  \label{exa:trivial}
  If $R$ is regular, then there are isomorphisms,
  \begin{displaymath}
    \K_1(\mod\,R) \cong \K_1(\proj\,R) \cong \K_1^{\mathrm{C}}(R) \cong R^*\,.
  \end{displaymath}
  The first isomorphism is by Quillen's resolution theorem \cite[\S4
  thm.~3]{MR0338129}, the second one is mentioned in \resref{Quillen},
  and the third one is well-known; see
  e.g.~\cite[exa.~(1.6)]{MR1382659}. \thmref{1} confirms this result,
  indeed, as $M=R$ is a representation generator for $\MCM\,R =
  \proj\,R$ one has $\Aut_R(M)_\mathrm{ab} = R^*$. As there are no
  Auslander--Reiten sequences in this case, the subgroup $\Xi$ is
  generated by the empty set, so $\Xi=0$.
\end{exa}

We now illustrate how \thmref{1} applies to compute $\K_1(\mod\,R)$
for the ring \mbox{$R=k[X]/(X^2)$}.  The answer is well-known to be
$k^*$, indeed, for any commutative artinian local ring $R$ with
residue field $k$ one has \mbox{$\K_1(\mod\,R) \cong k^*$} by
\cite[\S5 cor.~1]{MR0338129}.

\begin{exa}
  \label{exa:dual-numbers}
  \textsl{Let \mbox{$R=k[X]/(X^2)$} be the ring of dual numbers over a
    field $k$ with \mbox{$\operatorname{char}(k) \neq 2$}. Denote by
    \,\mbox{$\mathrm{inc} \colon \proj\,R \to \mod\,R$} the inclusion
    functor. The ho\-mo\-morphism $\K_1(\mathrm{inc})$ may be identified
    with the map,
  \begin{displaymath}
    \mu \colon R^* \longrightarrow k^*
    \qquad \text{given by} \qquad
    a+bX \longmapsto a^2\;.
  \end{displaymath}}
\end{exa}

\begin{proof}
  The maximal ideal \mbox{$\m=(X)$} is the only
  non-free indecomposable maximal Cohen--Macaulay $R$-module, so $M=R \oplus \m$
  is a representation generator for $\MCM\,R$; see \eqref{M}. There is an isomorphism
  \mbox{$k \to \End_R(\m)$} of $R$-algebras given by \mbox{$a \mapsto
    a1_\m$}, in particular, $\End_R(\m)$ is commutative.  Via this
  isomorphism, $k^*$ corresponds to $\Aut_R(\m)$. The
  Auslander--Reiten sequence ending in $\m$ is
  \begin{displaymath}
      0 \longrightarrow \m \stackrel{\iota}{\longrightarrow} R
      \stackrel{X}{\longrightarrow} \m \longrightarrow 0\;,
  \end{displaymath}
  where $\iota$ is the inclusion. The Auslander--Reiten homomorphism
  \smash{$\Upsilon = \binom{-1}{2} \colon \mathbb{Z} \to
    \mathbb{Z}^2$} is injective, so \thmref{1} can be applied.  Note
  that for every $a1_\m \in \Aut_R(\m)$, where $a \in k^*$, there is a
  commutative diagram,
  \begin{displaymath}
    \xymatrix{
      0 \ar[r] & \m \ar[d]^-{a1_\m}_-{\cong} \ar[r]^-{\iota} & R
      \ar[d]^-{a1_R}_-{\cong} 
      \ar[r]^-{X} & 
      \m \ar[d]^-{a1_\m}_-{\cong} \ar[r] & 0\;\phantom{.}
      \\
      0 \ar[r] & \m \ar[r]^-{\iota} &
      R \ar[r]^-{X} & \m \ar[r] & 0\;. 
    }
  \end{displaymath}
  Applying the tilde construction \resref{tilde} to the automorphisms
  $a1_\m$ and $a1_R$ one gets
  \begin{displaymath}
    \widetilde{a1_\m} = 
    \begin{pmatrix}
      1_R & 0 \\
      0 & a1_\m
    \end{pmatrix}
    \qquad \text{and} \qquad
    \widetilde{a1_R} = 
    \begin{pmatrix}
      a1_R & 0 \\
      0 & 1_\m
    \end{pmatrix};    
  \end{displaymath}
  see \exaref{diag}. In view of \dfnref{Xi} and
  \rmkref{remark-to-definition-Xi}, the subgroup $\Xi$ of $\Aut_R(R
  \oplus \m)_\mathrm{ab}$ is therefore generated by all elements of
  the form
  \begin{displaymath}
    \xi_a := 
    (\widetilde{a1_\m})(\widetilde{a1_R})^{-1}(\widetilde{a1_\m}) = 
    \begin{pmatrix}
      a^{-1}1_R & 0 \\
      0 & a^21_\m      
    \end{pmatrix}
    \qquad \text{where} \qquad
    a \in k^*\;.
  \end{displaymath}
  Denote by $\omega$ the composite of the isomorphisms,
  \begin{displaymath}
    \xymatrix{
    \Aut_R(R \oplus \m)_\mathrm{ab} \ar[r]^-{\delta}_-{\cong} &
    k^* \oplus \Aut_R(\m) \ar[r]_-{\cong} & k^* \oplus k^*\;,
    }
  \end{displaymath}
  where $\delta$ is the isomorphism from \prpref{Aut-ab-exa}. As
  $\omega(\xi_a) = (a^{-1},a)$ we get that $\omega(\Xi) = \{(a^{-1},a)
  \,|\, a \in k^*\}$ and thus $\omega$ induces the first group
  isomorphism below,
  \begin{displaymath}
    \xymatrix@C=2.4pc{
    \Aut_R(R \oplus \m)_\mathrm{ab}/\Xi \mspace{1mu}
    \ar[r]^-{\overline{\omega}}_-{\cong} & 
    \mspace{1mu} (k^* \oplus k^*)/\omega(\Xi) \mspace{1mu}
    \ar[r]^-{\chi}_-{\cong} & \mspace{1mu} k^*\;;
    }
  \end{displaymath}
  the second isomorphism is induced by the surjective homomorphism
  \mbox{$k^* \oplus k^* \to k^*$}, given by \mbox{$(b,a) \mapsto ba$},
  whose kernel is exactly $\omega(\Xi)$. In view of \thmref{1} and the
  isomorphisms $\overline{\omega}$ and $\chi$ above, it follows that
  $\K_1(\mod\,R) \cong k^*$.

  \thmref{1} asserts that $\K_1(\mathrm{inc})$ may be identified with
  the homomorphism
  \begin{displaymath}
    \lambda \colon R^* \longrightarrow \Aut_R(R \oplus \m)_{\mathrm{ab}}/\Xi
    \qquad \text{given by} \qquad
    r \longmapsto
    \begin{pmatrix}
      r1_R & 0 \\
      0 & 1_\m
    \end{pmatrix}.
  \end{displaymath}
  It remains to note that the isomorphism $\chi\overline{\omega}$ identifies
  $\lambda$ with the homomorphism~$\mu$ described in the example,
  indeed, one has $\chi\overline{\omega}\lambda = \mu$.
\end{proof}

\exaref{dual-numbers} shows that for \mbox{$R=k[X]/(X^2)$} the
canonical homomorphism,
\begin{displaymath}
  \xymatrix@C=3.5pc{
  R^* \cong \K_1(\proj\,R)
  \ar[r]^-{\K_1(\mathrm{inc})} & \K_1(\mod\,R) \cong k^*
  },
\end{displaymath}
is not an isomorphism. It turns out that if $k$ is algebraically
closed with characteristic zero, then there exists a non-canonical
isomorphism between $R^*$ and $k^*$.

\begin{prp}
  \label{prp:char}
  Let \mbox{$R=k[X]/(X^2)$} where $k$ is an algebraically closed field
  with characteristic $p \geqslant 0$. The following assertions hold.
  \begin{prt}
  \item If $p>0$, then the groups $R^*$ and $k^*$ are not isomorphic.
  \item If $p=0$, then there exists a (non-canonical) group isomorphism
    $R^* \cong k^*$.
  \end{prt}
\end{prp}

\begin{proof}
  There is a group isomorphism $R^* \to k^* \oplus k^+$ given by $a+bX
  \mapsto (a,b/a)$, where $k^+$ denotes the underlying abelian group
  of the field $k$.

  \proofof{(a)} Let \mbox{$\varphi=(\varphi_1,\varphi_2) \colon k^*
    \to k^* \oplus k^+$} be any group homomorphism.  As $k$ is
  algebraically closed, every element in $x \in k^*$ has the form $x =
  y^p$ for some $y \in k^*$. Therefore \mbox{$\varphi(x) = \varphi(y^p) =
    \varphi(y)^p = (\varphi_1(y),\varphi_2(y))^p =
    (\varphi_1(y)^p,p\varphi_2(y)) = (\varphi_1(x),0)$}, which shows
  that $\varphi$ is not surjective.

  \proofof{(b)} Since $p=0$ the abelian group $k^+$ is divisible and
  torsion free. Therefore \mbox{$k^+ \cong \mathbb{Q}^{(I)}$} for some
  index set $I$. There exist algebraic field extensions of
  $\mathbb{Q}$ of any finite degree, and these are all contained in
  the algebraically closed field $k$. Thus \mbox{$|I| =
    \dim_\mathbb{Q}k$} must be infinite, and it follows that $|I| =
  |k|$.

  The abelian group $k^*$ is also divisible, but it has torsion. Write
  $k^* \cong T \oplus (k^*/T)$, where \mbox{$T = \{x \in k^* \,|\,
    \exists\, n \in \mathbb{N} \colon x^n=1 \}$} is the torsion
  subgroup of $k^*$. For the divisible torsion free abelian group
  $k^*/T$ one has \mbox{$k^*/T \cong \mathbb{Q}^{(J)}$} for some index
  set $J$. It is not hard to see that $|J|$ must be infinite, and
  hence $|J| = |k^*/T|$. As $|T| = \aleph_0$ it follows that $|k| =
  |k^*| = \aleph_0 + |J| = |J|$.

  Since $|J|=|k|=|I|$ one gets \mbox{$k^* \cong T \oplus
  \mathbb{Q}^{(J)} \cong T \oplus \mathbb{Q}^{(J)} \oplus
  \mathbb{Q}^{(I)} \cong k^* \oplus k^+$}.
\end{proof}

The artinian ring \mbox{$R=k[X]/(X^2)$} from \exaref{dual-numbers} has
length $\ell=2$ and this power is also involved in the description of
the homomorphism $\mu = \K_1(\mathrm{inc})$.  The next result shows that this is no coincidence.
As \prpref{artinian} might be well-known to experts, and since we do not really need it, we do not give a proof. 

\begin{prp}
  \label{prp:artinian}
  Let $(R,\m,k)$ be a commutative artinian local ring of
  length~$\ell$. The group homomorphism \mbox{$R^* \cong
    \K_1(\proj\,R) \to \K_1(\mod\,R) \cong k^*$} induced by the
  inclusion \mbox{$\mathrm{inc} \colon \proj\,R \to \mod\,R$} is the
  composition of the homomorphisms,
  \begin{displaymath}
    \xymatrix{
      R^* \ar[r]^-{\pi} & k^* \ar[r]^-{(\cdot)^\ell} & k^*\;,
    }
  \end{displaymath}
  where \mbox{$\pi \colon R \twoheadrightarrow R/\m=k$} is the
  canonical quotient map and $(\cdot)^\ell$ is the $\ell$'th power.
\end{prp}

Our next example is a non-artinian ring, namely the simple curve
singularity of type ($A_2$) studied by e.g.~Herzog \cite[Satz~1.6]{Herzog} and
Yoshino \cite[prop.~(5.11)]{yos}.

\begin{exa}
  \label{exa:simple-curve-singularity}
  \textsl{Let \mbox{$R=k[\hspace*{-1.4pt}[T^2,T^3]\hspace*{-1.4pt}]$}
    where $k$ is an al\-ge\-bra\-i\-cal\-ly closed field with
    \mbox{$\operatorname{char}(k) \neq 2$}.  Denote by
    \,\mbox{$\mathrm{inc} \colon \proj\,R \to \mod\,R$}
    the inclusion functor.  The homomorphism $\K_1(\mathrm{inc})$ may
    be identified with the inclusion map,
    \begin{displaymath} 
      \mu \colon R^*=k[\hspace*{-1.4pt}[T^2,T^3]\hspace*{-1.4pt}]^* 
      \hookrightarrow
      k[\hspace*{-1.4pt}[T]\hspace*{-1.4pt}]^*\;.
    \end{displaymath}}
\end{exa}

\begin{proof}
  The maximal ideal $\m =
  (T^2,T^3)$ is the only non-free indecomposable maximal Cohen--Macaulay
  $R$-module, so $M=R \oplus \m$ is a representation generator for
  $\MCM\,R$; see~\eqref{M}. Even though $T$ is not an element in 
  $R=k[\hspace*{-1.4pt}[T^2,T^3]\hspace*{-1.4pt}]$, multiplication by
  $T$ is a well-defined endomorphism of $\m$. Thus there is a
  ring homomorphism,
  \begin{displaymath}
    \chi \colon k[\hspace*{-1.4pt}[T]\hspace*{-1.4pt}] \longrightarrow
    \End_R(\m) 
    \qquad \text{given by} \qquad h \longmapsto h1_\m\;.
  \end{displaymath}
  It is not hard to see that $\chi$ is injective. To prove that it is surjective, i.e.~that one has \mbox{$\End_R(\m)/k[\hspace*{-1.4pt}[T]\hspace*{-1.4pt}]=0$}, note that there is a short exact sequence of $R$-modules, 
  \begin{displaymath}
    0 \longrightarrow k[\hspace*{-1.4pt}[T]\hspace*{-1.4pt}]/R 
    \longrightarrow \End_R(\m)/R
    \longrightarrow \End_R(\m)/k[\hspace*{-1.4pt}[T]\hspace*{-1.4pt}]
    \longrightarrow 0\;.
  \end{displaymath}
  To see that \mbox{$\End_R(\m)/k[\hspace*{-1.4pt}[T]\hspace*{-1.4pt}]=0$}, it suffices to argue that the $R$-module $\End_R(\m)/R$ is simple. As noted in the beginning of the proof of \prpref{Aut-ab-exa}, the inclusion $\m \hookrightarrow R$ induces an isomorphism $\End_R(\m) \cong \Hom_R(\m,R)$,
so by applying $\Hom_R(-,R)$ to the short exact sequence $0 \to \m \to R \to k \to 0$, it follows that 
  \begin{displaymath}
     \End_R(\m)/R \cong \mathrm{Ext}_R^1(k,R)\;.
  \end{displaymath}
  The latter module is isomorphic to $k$ since $R$ is a $1$-dimensional Gorenstein ring.

  Note that via the isomorphism $\chi$, the group \mbox{$k[\hspace*{-1.4pt}[T]\hspace*{-1.4pt}]^*$} corresponds to $\Aut_R(\m)$.
  
  The Auslander--Reiten sequence ending in $\m$ is
  \begin{displaymath}
    \xymatrix@C=8ex{
      0 \ar[r] & \m \ar[r]^-{(1 \ \, -T)^t}
      & R \oplus \m \ar[r]^-{(T^2 \ \, T)} & \m \ar[r] & 0\;.
    }
  \end{displaymath}
 Since the Auslander--Reiten homomorphism
  \smash{$\Upsilon = \binom{-1}{1} \colon \mathbb{Z} \to
    \mathbb{Z}^2$} is injective, \thmref{1} can be applied. 
  We regard elements in $R \oplus \m$ as column vectors.  Let $\alpha
  = h1_\m \in \Aut_R(\m)$, where \mbox{$h \in
    k[\hspace*{-1.4pt}[T]\hspace*{-1.4pt}]^*$}, be given. Write $h =
  f+gT$ for some $f \in R^*$ and $g \in R$. It is straight\-forward to
  verify that there is a commutative diagram,
  \begin{displaymath}
    \xymatrix@C=4.5pc{
      0 \ar[r] & \m \ar[d]_{\cong}^-{\gamma \,=\, (f-gT)1_\m} \ar[r]^-{(1 \ \, -T)^t}
      & R \oplus \m \ar[d]_{\cong}^-{\text{\scriptsize 
      $\beta =\! 
      \begin{pmatrix}
        f \!\! & \!\! g \\
        gT^2 \!\! & \!\! f        
      \end{pmatrix}
      $}}
 \ar[r]^-{(T^2 \ \, T)} & \m \ar[d]_{\cong}^-{\alpha \,=\, (f+gT)1_\m} \ar[r] & 0\;\phantom{.} \\
      0 \ar[r] & \m \ar[r]_-{(1 \ \, -T)^t}
      & R \oplus \m \ar[r]_-{(T^2 \ \, T)} & \m \ar[r] & 0\;.
    }
  \end{displaymath}
  Note that $\beta$ really is an automorphism; indeed, its inverse is
  given by
  \begin{displaymath}
    \beta^{-1} = (f^2-g^2T^2)^{-1}
      \begin{pmatrix}
        f \!\! & \!\! -g \\
        -gT^2 \!\! & \!\! f        
      \end{pmatrix}.
  \end{displaymath}
  We now apply the tilde construction \resref{tilde} to 
  $\alpha$, $\beta$, and $\gamma$; by \exaref{diag} we get:
  \begin{displaymath}
    \tilde{\alpha} = 
    \begin{pmatrix}
      1 & 0 \\
      0 & f+gT
    \end{pmatrix}
    \,, \quad
    \tilde{\beta} = \beta
    , \quad \text{and} \quad
    \tilde{\gamma} = 
    \begin{pmatrix}
      1 & 0 \\
      0 & f-gT
    \end{pmatrix}.
  \end{displaymath}
  In view of \dfnref{Xi} and \rmkref{remark-to-definition-Xi}, the
  subgroup $\Xi$ of $\Aut_R(R \oplus \m)_\mathrm{ab}$ is therefore
  generated by all the elements
  \begin{displaymath}
    \xi_h := \tilde{\alpha}\tilde{\beta}^{-1}\tilde{\gamma}
    = (f^2-g^2T^2)^{-1}
    \begin{pmatrix}
      f & -g(f-gT) \\
      -gT^2(f+gT) & f(f^2-g^2T^2)
    \end{pmatrix}.
  \end{displaymath}
  
  Denote by $\omega$ the composite of the isomorphisms,
  \begin{displaymath}
    \xymatrix@C=3pc{
    \Aut_R(R \oplus \m)_\mathrm{ab} \ar[r]^-{\delta}_-{\cong} &
    k^* \oplus \Aut_R(\m) \ar[r]^-{1\, \oplus\, \chi^{-1}}_-{\cong} & 
    k^* \oplus k[\hspace*{-1.4pt}[T]\hspace*{-1.4pt}]^*\;,
    }
  \end{displaymath}
  where $\delta$ is the isomorphism from \prpref{Aut-ab-exa}.  Note
  that $\delta(\xi_h) = ([f]_\m,1_\m) = (h(0),1_\m)$ and hence
  $\omega(\xi_h) = (h(0),1)$. It follows that $\omega(\Xi) = k^*
  \oplus \{1\}$ and thus $\omega$ induces a group
  isomorphism,
  \begin{displaymath}
    \overline{\omega} \colon
    \Aut_R(R \oplus \m)_\mathrm{ab}/\Xi \mspace{1mu}
    \stackrel{\cong}{\longrightarrow}
    (k^* \oplus k[\hspace*{-1.4pt}[T]\hspace*{-1.4pt}]^*)/\omega(\Xi) =
    \mspace{1mu} k[\hspace*{-1.4pt}[T]\hspace*{-1.4pt}]^*\;.
  \end{displaymath}
  In view of this isomorphism, \thmref{1} shows that
  \mbox{$\K_1(\mod\,R) \cong
    k[\hspace*{-1.4pt}[T]\hspace*{-1.4pt}]^*$}.  \thmref{1} also
  asserts that $\K_1(\mathrm{inc})$ may be identified with the
  homomorphism
  \begin{displaymath}
    \lambda \colon R^* \longrightarrow \Aut_R(R \oplus \m)_{\mathrm{ab}}/\Xi
    \qquad \text{given by} \qquad
    f \longmapsto
    \begin{pmatrix}
      f1_R & 0 \\
      0 & 1_\m
    \end{pmatrix}.
  \end{displaymath}
  It remains to note that the isomorphism $\overline{\omega}$ identifies
  $\lambda$ with the inclusion map $\mu$ described in the example,
  indeed, one has $\overline{\omega}\lambda = \mu$.
\end{proof}

\section*{Acknowledgements}

It is a pleasure to thank Peter J{\o}rgensen without whom this paper
could not have been written. However, Peter did not wish to be a
coauthor of this manuscript. We are also grateful to Marcel
B{\"o}kstedt and Charles A. Weibel for
valuable input on the Gersten--Sherman transformation, and to
Christian U.~Jensen for pointing out \prpref{char}. \new{Finally, we thank Viraj Navkal and the anonymous referee for their thoughtful comments and for making us aware of the paper \cite{Navkal}.}

\def\cprime{$'$}
  \newcommand{\arxiv}[2][AC]{\mbox{\href{http://arxiv.org/abs/#2}{\sf arXiv:#2
  [math.#1]}}}
  \newcommand{\oldarxiv}[2][AC]{\mbox{\href{http://arxiv.org/abs/math/#2}{\sf
  arXiv:math/#2
  [math.#1]}}}\providecommand{\MR}[1]{\mbox{\href{http://www.ams.org/mathscinet-getitem?mr=#1}{#1}}}
  \renewcommand{\MR}[1]{\mbox{\href{http://www.ams.org/mathscinet-getitem?mr=#1}{#1}}}
\providecommand{\bysame}{\leavevmode\hbox to3em{\hrulefill}\thinspace}
\providecommand{\MR}{\relax\ifhmode\unskip\space\fi MR }
\providecommand{\MRhref}[2]{%
  \href{http://www.ams.org/mathscinet-getitem?mr=#1}{#2}
}
\providecommand{\href}[2]{#2}


\begin{thebibliography}{10}

\bibitem{MR0349747}
Maurice Auslander, \emph{Representation theory of {A}rtin algebras. {I}, {II}},
  Comm. Algebra \textbf{1} (1974), 177--268; ibid. 1 (1974), 269--310.

\bibitem{MR816307}
\bysame, \emph{Rational singularities and almost split sequences}, Trans. Amer.
  Math. Soc. \textbf{293} (1986), no.~2, 511--531. 

\bibitem{MAsROB89}
Maurice Auslander and Ragnar-Olaf Buchweitz, \emph{The homological theory of
  maximal {C}ohen-{M}acaulay approximations}, M\'em. Soc. Math. France (N.S.)
  (1989), no.~38, 5--37, Colloque en l'honneur de Pierre Samuel (Orsay, 1987).

\bibitem{MR816889}
Maurice Auslander and Idun Reiten, \emph{Grothendieck groups of algebras and
  orders}, J. Pure Appl. Algebra \textbf{39} (1986), no.~1-2, 1--51.

\bibitem{AR}
\bysame, \emph{Almost split sequences for rational double points}, Trans. Amer.
  Math. Soc. \textbf{302} (1987), no.~1, 87--97. 

\bibitem{rta}
Maurice Auslander, Idun Reiten, and Sverre~O. Smal{\o}, \emph{Representation
  theory of {A}rtin algebras}, Cambridge Studies in Advanced Mathematics,
  vol.~36, Cambridge University Press, Cambridge, 1995. 

\bibitem{MR0249491}
Hyman Bass, \emph{Algebraic {$K$}-theory}, W. A. Benjamin, Inc., New
  York-Amsterdam, 1968. 

\bibitem{Benson}
David~J. Benson, \emph{Representations and cohomology. {I}}, second ed.,
  Cambridge Studies in Advanced Mathematics, vol.~30, Cambridge University
  Press, Cambridge, 1998, Basic representation theory of finite groups and
  associative algebras. 

\bibitem{Butler}
Michael C.~R. Butler, \emph{Grothendieck groups and almost split sequences},
  Integral representations and applications ({O}berwolfach, 1980), Lecture
  Notes in Math., vol. 882, Springer, Berlin, 1981, pp.~357--368. 

\bibitem{Cimen1}
Nuri Cimen, \emph{One-dimensional rings of finite {C}ohen-{M}acaulay type},
  ProQuest LLC, Ann Arbor, MI, 1994, Thesis (Ph.D.)--The University of Nebraska
  - Lincoln. 

\bibitem{Cimen2}
\bysame, \emph{One-dimensional rings of finite {C}ohen-{M}acaulay type}, J.
  Pure Appl. Algebra \textbf{132} (1998), no.~3, 275--308. 

\bibitem{DrozdRoiter}
Ju.~A. Drozd and Andrei~V. Ro{\u\i}ter, \emph{Commutative rings with a finite
  number of indecomposable integral representations}, Izv. Akad. Nauk SSSR Ser.
  Mat. \textbf{31} (1967), 783--798. 

\bibitem{rha}
Edgar~E. Enochs and Overtoun M.~G. Jenda, \emph{Relative homological algebra},
  de Gruyter Expositions in Mathematics, vol.~30, Walter de Gruyter \& Co.,
  Berlin, 2000. 

\bibitem{Esnault}
H{\'e}l{\`e}ne Esnault, \emph{Reflexive modules on quotient surface
  singularities}, J. Reine Angew. Math. \textbf{362} (1985), 63--71. 

\bibitem{FuHa}
William Fulton and Joe Harris, \emph{Representation theory}, Graduate Texts in
  Mathematics, vol. 129, Springer-Verlag, New York, 1991, A first course,
  Readings in Mathematics. 

\bibitem{PGb62}
Pierre Gabriel, \emph{Des cat\'egories ab\'eliennes}, Bull. Soc. Math. France
  \textbf{90} (1962), 323--448. 

\bibitem{MR0382398}
Stephen~M. Gersten, \emph{Higher {$K$}-theory of rings}, Algebraic
  {$K$}-theory, {I}: {H}igher {$K$}-theories ({P}roc. {C}onf. {S}eattle {R}es.
  {C}enter, {B}attelle {M}emorial {I}nst., 1972), Springer, Berlin, 1973,
  pp.~3--42. Lecture Notes in Math., Vol. 341. 

\bibitem{GreenReiner}
Edward~L. Green and Irving Reiner, \emph{Integral representations and
  diagrams}, Michigan Math. J. \textbf{25} (1978), no.~1, 53--84. 

\bibitem{Herzog}
J{\"u}rgen Herzog, \emph{Ringe mit nur endlich vielen {I}somorphieklassen von
  maximalen, unzerlegbaren {C}ohen-{M}acaulay-{M}oduln}, Math. Ann.
  \textbf{233} (1978), no.~1, 21--34. 

\bibitem{MR1838439}
Tsit-Yuen Lam, \emph{A first course in noncommutative rings}, second ed.,
  Graduate Texts in Mathematics, vol. 131, Springer-Verlag, New York, 2001.

\bibitem{GJL07}
Graham~J. Leuschke, \emph{Endomorphism rings of finite global dimension},
  Canad. J. Math. \textbf{59} (2007), no.~2, 332--342. 

\bibitem{MR0068541}
Oscar Litoff, \emph{On the commutator subgroup of the general linear group},
  Proc. Amer. Math. Soc. \textbf{6} (1955), 465--470. 

\bibitem{Navkal}
Viraj Navkal, \emph{K'-theory of a local ring of finite {C}ohen-{M}acaulay
  type}, preprint (2012) \arxiv[KT]{1108.2000v2} (see \href{http://www.math.ucla.edu/~viraj/}{\sf http://www.math.ucla.edu/\~{}viraj/} for the latest version).

\bibitem{MR0338129}
Daniel Quillen, \emph{Higher algebraic {$K$}-theory. {I}}, Algebraic
  {$K$}-theory, {I}: {H}igher {$K$}-theories ({P}roc. {C}onf., {B}attelle
  {M}emorial {I}nst., {S}eattle, {W}ash., 1972), Springer, Berlin, 1973,
  pp.~85--147. Lecture Notes in Math., Vol. 341. 

\bibitem{MR1282290}
Jonathan Rosenberg, \emph{Algebraic {$K$}-theory and its applications},
  Graduate Texts in Mathematics, vol. 147, Springer-Verlag, New York, 1994.

\bibitem{Sherman}
Clayton Sherman, \emph{Group representations and algebraic {$K$}-theory},
  Algebraic {$K$}-theory, {P}art {I} ({O}berwolfach, 1980), Lecture Notes in
  Math., vol. 966, Springer, Berlin, 1982, pp.~208--243. 

\bibitem{MR1382659}
Vasudevan Srinivas, \emph{Algebraic {$K$}-theory}, second ed., Progress in
  Mathematics, vol.~90, Birkh\"auser Boston Inc., Boston, MA, 1996.

\bibitem{MR0267009}
Leonid~N. Vaserstein, \emph{On the stabilization of the general linear group
  over a ring}, Math. USSR-Sb. \textbf{8} (1969), 383--400. 

\bibitem{MR2111217}
\bysame, \emph{On the {W}hitehead determinant for semi-local rings}, J. Algebra
  \textbf{283} (2005), no.~2, 690--699. 

\bibitem{Wiegand1}
Roger Wiegand, \emph{Noetherian rings of bounded representation type},
  Commutative algebra ({B}erkeley, {CA}, 1987), Math. Sci. Res. Inst. Publ.,
  vol.~15, Springer, New York, 1989, pp.~497--516. 

\bibitem{Wiegand2}
\bysame, \emph{One-dimensional local rings with finite {C}ohen-{M}acaulay
  type}, Algebraic geometry and its applications ({W}est {L}afayette, {IN},
  1990), Springer, New York, 1994, pp.~381--389. 

\bibitem{yos}
Yuji Yoshino, \emph{Cohen-{M}acaulay modules over {C}ohen-{M}acaulay rings},
  London Mathematical Society Lecture Note Series, vol. 146, Cambridge
  University Press, Cambridge, 1990. 

\end{thebibliography}

\end{document}